\documentclass[a4paper,11pt]{article}

\usepackage{geometry}
\usepackage{latexsym}
\usepackage{mathrsfs}
\usepackage{amsfonts}
\usepackage{amssymb}
\usepackage{subfigure}    %for parallel figures
\usepackage{graphicx}
\usepackage{epstopdf}
\usepackage{float}
\usepackage{multirow}
\usepackage{booktabs}
\usepackage{bm}
\usepackage{etoolbox}
\usepackage{amsmath}
\usepackage{amsthm}
\usepackage{color}
\usepackage[subnum]{cases}
\usepackage{rotating}
\usepackage{enumitem}
\usepackage{accents}
\usepackage{algorithm}
\usepackage{algorithmic}
\usepackage[title]{appendix}
\usepackage{mathtools}
\usepackage{caption}
\usepackage{prettyref}
\usepackage{tabularx}
\usepackage{booktabs}
\usepackage{stmaryrd}
\usepackage{url}
\usepackage{listings}
\usepackage{tcolorbox}
\usepackage{titlefoot}
\makeatletter
\@addtoreset{equation}{section}
\makeatother

\makeatletter
\def \wideubar{\underaccent{{\cc@style\underline{\mskip15mu}}}}
\def \widebar{\accentset{{\cc@style\underline{\mskip10mu}}}}
\makeatother

\definecolor{blue}{rgb}{0,0,0.9}
\definecolor{red}{rgb}{0.9,0,0}
\definecolor{green}{rgb}{0,0.9,0}
\definecolor{lightgreen}{rgb}{0.1,0.5,0.1}

\usepackage[colorlinks=true,
breaklinks=true,
bookmarks=true,
urlcolor=blue,
citecolor=lightgreen,
linkcolor=lightgreen,
bookmarksopen=false,
draft=false]{hyperref}

\graphicspath{{figures/}}

\begin{document}

\newtheorem{property}{Property}[section]
\newtheorem{proposition}{Proposition}[section]
\newtheorem{append}{Appendix}[section]
\newtheorem{definition}{Definition}[section]
\newtheorem{lemma}{Lemma}[section]
\newtheorem{corollary}{Corollary}[section]
\newtheorem{theorem}{Theorem}[section]
\newtheorem{remark}{Remark}[section]
\newtheorem{problem}{Problem}[section]
\newtheorem{example}{Example}[section]
\newtheorem{assumption}{Assumption}
\renewcommand*{\theassumption}{\Alph{assumption}}

% \title{Convergence Analysis of the ripALM for Solving General Constrained Convex Optimization Problems}

\title{Convergence Analysis of a Relative-type Inexact Preconditioned Proximal ALM for Convex Nonlinear Programming}

\author{
Lei Yang\thanks{School of Computer Science and Engineering, and Guangdong Province Key Laboratory of Computational Science, Sun Yat-sen University ({\tt yanglei39@mail.sysu.edu.cn}). },~~
Jiayi Zhu\thanks{School of Computer Science and Engineering, Sun Yat-sen University ({\tt zhujy86@mail2.sysu.edu.cn}).}, ~~
Ling Liang\thanks{(Corresponding author) Department of Mathematics, The University of Tennessee, Knoxville ({\tt liang.ling@u.nus.edu}).},~~
Kim-Chuan Toh\thanks{Department of Mathematics, and Institute of Operations Research and Analytics, National University of Singapore ({\tt mattohkc@nus.edu.sg}).}
}

% \author{Jiayi Zhu, ~~Ling Liang, ~~Lei Yang, ~~Kim-Chuan Toh}

%\date{Started on January 17, 2022; Re-started on January 1, 2024}
\date{Updated on \today}

\maketitle

\begin{abstract}
This article investigates the convergence properties of a relative-type inexact preconditioned proximal augmented Lagrangian method (rip$^2$ALM) for convex nonlinear programming, a fundamental class of optimization problems with broad applications in science and engineering. Inexact proximal augmented Lagrangian methods have proven to be highly effective for solving such problems, owing to their attractive theoretical properties and strong practical performance. However, the convergence behavior of the relative-type inexact preconditioned variant remains insufficiently understood. This work aims to reduce this gap by rigorously establishing the global convergence of the sequence generated by rip$^2$ALM and proving its asymptotic (super)linear convergence rate under standard assumptions. In addition, we derive the global ergodic convergence rate with respect to both the primal feasibility violation and the primal objective residual, thereby offering a more comprehensive understanding of the overall performance of rip$^2$ALM. These results deepen our theoretical understanding of the family of proximal augmented Lagrangian methods and motivate their development for practical, large-scale structured application problems.

\vspace{5mm}

\vspace{2mm}
\end{abstract}

%%%%%%%%%%%%%%%%%%%%%%%%%%%%%%%%%%%%%%%%%%%%%%%%%
\section{Introduction}

A fundamental problem in mathematical optimization is to minimize a convex objective subject to linear equality and convex nonlinear inequality constraints, formulated as:
\begin{equation}\label{gen-mainprob}
\begin{aligned}
\min_{\bm{x}\in\mathbb{R}^N} \quad f(\bm{x}), \quad \mathrm{s.t.} \quad A\bm{x} = \bm{b}, ~~ g(\bm{x}) \leq \bm{0},
\end{aligned}
\end{equation}
where $f:\mathbb{R}^{N}\to\mathbb{R}\cup\{+\infty\}$ is a proper closed (possibly nonsmooth) convex function, $A\in \mathbb{R}^{m_{1}\times N}$ and $\bm{b}\in\mathbb{R}^{m_{1}}$ are given data, and $g(\bm{x}) = \big(g_1(\bm{x}), \cdots, g_{m_{2}}(\bm{x})\big)$ with each $g_i:\mathbb{R}^{N}\to\mathbb{R}$ ($i=1,\cdots,m_2$) being a continuously differentiable convex function.

Problem \eqref{gen-mainprob} underlies numerous fundamental models in contemporary application areas such as machine learning, signal/image processing, and data science. For example, the support vector machine can be formulated as a convex quadratic programming problem with linear inequality constraints \cite{boyd2004convex,hearst1998support}. The basis pursuit denoising problem, which is a central problem in compressed sensing, seeks to recover a sparse signal from noisy linear measurements by minimizing the $\ell_1$-norm subject to convex nonlinear inequality constraints \cite{van2009probing}. Regularized logistic and least-squares regression can also be cast as convex constrained problems, where a convex loss function is minimized under suitable regularization constraints \cite{shalev2014understanding,zafar2019fairness}. Therefore, the general convex nonlinear programming problem \eqref{gen-mainprob} not only provides a unifying framework for modeling diverse applications, but also serves as the foundation for developing reliable algorithms. This makes it a cornerstone for both theoretical analysis and algorithmic development in optimization and operations research.

%While first-order methods are simple and easy to implement, they are generally less effective when high accuracy or fast convergence is required, particularly compared with higher-order methods.
Existing popular approaches for solving problem \eqref{gen-mainprob} can be broadly categorized into three classes: \textit{penalty methods}, \textit{interior-point methods}, and \textit{sequential quadratic programming (SQP) methods}. Penalty methods solve the original constrained problem by addressing a sequence of unconstrained problems, which are obtained by incorporating penalty terms for constraint violations in the objective, using either differentiable \cite{di1985continuously,di1989exact,glad1979multiplier,zavala2014scalable} or nondifferentiable \cite{conn1973constrained,fletcher2000practical,zangwill1967non} penalty functions. Interior-point methods \cite{nesterov1994interior,polik2010interior} replace inequality constraints with barrier functions in the objective, enabling the use of the highly efficient Newton's method on the resulting equality-constrained formulation. They exhibit strong practical performance and serve as the basis for state-of-the-art solvers such as KNITRO \cite{byrd1999interior} and IPOPT \cite{wachter2006implementation}. SQP methods \cite{gill2005snopt,nocedal2006numerical} iteratively solve a sequence of quadratic subproblems, which are obtained by linearizing the nonlinear constraints and approximating the objective via a local quadratic proxy. By leveraging second-order information, SQP methods can achieve high accuracy and rapid local convergence rates. However, they are often computationally expensive and sensitive to problems' conditioning, requiring accurate derivatives and careful globalization strategies to ensure robustness from poor initialization. Due to the extensive literature and numerous variants of these methods, a comprehensive survey is beyond the scope of this work.

Augmented Lagrangian methods (ALMs) \cite{hestenes1969multiplier,r1976augmented}, as an important class of exact penalty methods \cite[Chapter 17]{nocedal2006numerical}, are among the most widely used approaches for constrained optimization, because of their appealing convergence properties and strong practical performance. In recent years, they have been successfully applied to a wide range of large-scale problems, including statistical regression and conic programming, often demonstrating superior efficiency compared with other state-of-the-art alternatives; see, e.g., \cite{lst2018highly,li2018qsdpnal,liang2021inexact,lin2019efficient,zst2010newton-cg} and references therein. Despite these advantages, ALMs may suffer from severe ill-conditioning (primarily due to degeneracy) as well as sensitivity to the updating strategy of the penalty parameter and to the error criteria used for the approximate minimization of the subproblems. On the theoretical side, significant efforts have been devoted to developing relaxed conditions under which ALMs can achieve fast asymptotic convergence, though such conditions are often unverifiable in practice \cite{cui2019r}. On the practical side, since ALM subproblems are rarely solved exactly, inexact variants are essential.

Recently, inexact proximal ALMs, which incorporate proximal terms into ALM subproblems, have attracted growing attention due to their ability to mitigate ill-conditioning of the subproblem and offer stronger convergence guarantees; see \cite{lst2020asymptotically,liang2022qppal,lin2019efficient,ylct2024corrected,zlyt2024ripalm} and references therein. While the convergence properties of the \textit{absolute-type} inexact proximal ALM (pALM) have been well studied \cite{lst2020asymptotically,r1976augmented}, the convergence behavior of its \textit{relative-type} counterpart for solving the general convex nonlinear programming problem \eqref{gen-mainprob} remains insufficiently understood. This article aims to reduce this gap by rigorously establishing the convergence properties of a relative-type inexact preconditioned proximal ALM (rip$^2$ALM) in the general setting. Here, the term ``\textit{absolute-type}" refers to the use of a predetermined tolerance sequence that controls the accuracy required for solving each pALM subproblem. Such a sequence is typically required to be summable and must be carefully chosen to avoid being either overly conservative or excessively aggressive. Consequently, obtaining satisfactory numerical performance often entails delicate parameter tuning, which may incur considerable effort and potential inefficiencies in practical implementations. In contrast, the \textit{relative-type} criterion adopts an adaptive strategy to control the inexactness in subproblem minimization based on computable quantities associated with the iterates of the algorithm. This approach eliminates the need for a summable tolerance sequence, thereby improving both the robustness and practical implementability of the algorithm.

{\bf Contributions.} In this work, we establish and extend the convergence properties of a relative-type inexact proximal ALM, originally proposed in our earlier work \cite{zlyt2024ripalm} for linearly constrained convex optimization, to the general convex nonlinear programming problem \eqref{gen-mainprob}. The proposed algorithmic framework also accommodates a preconditioned proximal term, which further enhances the versatility and flexibility of the method. We prove the global convergence of the sequence generated by rip$^2$ALM under mild assumptions and further establish its asymptotic (super)linear convergence rate under a standard error bound condition, thereby providing a refined characterization of the algorithm's asymptotic behavior. Moreover, we derive a novel global ergodic convergence rate with respect to both the primal feasibility violation and the primal objective residual. These results not only confirm the robustness of rip$^2$ALM but also deepen our understanding of the underlying trade-offs among efficiency, computational cost, and inexactness in subproblem minimization. Collectively, these findings enrich the convergence theory of rip$^2$ALM and provide practical guidance for its application to large-scale convex optimization problems. Finally, we remark that while our results are established for the constraints $-g(\bm{x})\in \mathbb{R}^{m_2}_+$, with suitable adaptation, they are also applicable to more general conic constraints of the form $-g(\bm{x})\in \mathcal{K}$, where $\mathcal{K}$ is a proper closed convex cone.

%\red{\fbox{Please check what conditions must the cone satisfy, e.g., $C^2$-reducible cone?}}}

The remainder of this article is organized as follows. Section \ref{sec:ripalm} introduces the proposed rip$^2$ALM, followed by a comprehensive convergence analysis in Section \ref{section:analysis}, where we establish its global convergence, asymptotic (super)linear convergence rate, and global ergodic convergence rate under suitable assumptions. Section \ref{sec-conclusion} concludes the paper with a summary of the main contributions and a discussion of potential directions for future research. All technical proofs are deferred to the appendices.

\vspace{2mm}
\textbf{Notation.}
%%%%%%%%%%%%%%%%%%%%%%%%%%%%%%%%%%%%%%%%%%%
%\section{Notation and preliminaries}\label{sec:preliminary}
%In this section, we introduce some background knowledge related to this work.
%\textbf{Notation.}
We use $\mathbb{R}^n$ ($\mathbb{R}_{+}^n$) and $\mathbb{R}^{m \times n}$ ($\mathbb{R}_{+}^{m \times n}$) to denote the sets of $n$-dimensional real (non-negative) vectors and $m \times n$ real (non-negative) matrices, respectively. For a vector $\bm{x}\in\mathbb{R}^n$,
%$x_i$ denotes its $i$-th entry,
$\|\bm{x}\|$ denotes its Euclidean norm and $\|\bm{x}\|_H:=\sqrt{\langle\bm{x},\, H\bm{x}\rangle}$ denotes its weighted norm associated with a symmetric positive definite matrix $H\in\mathbb{R}^{n\times n}$. %For simplicity, given an integer $n>0$, we use $\bm{1}_{n}\in\mathbb{R}^n$ to denote the $n$-dimensional vector of all ones, and use $I_n$ to denote the $n \times n$ identity matrix.
A \emph{multifunction} (also known as a set-valued mapping) is a generalization of the notion of a function in which each input may correspond to a set of outputs rather than a single value. Formally, a multifunction $F$ from a vector space $\mathbb{R}^n$ to the power set $2^{\mathbb{R}^m}$ is written as $F:\mathbb{R}^n \rightrightarrows \mathbb{R}^m$, where for each $\bm{x}\in\mathbb{R}^n$, the image $F(\bm{x})$ is a subset of $\mathbb{R}^m$.
%\textbf{Convex analysis.}
For an extended-real-valued function $f: \mathbb{R}^{n} \rightarrow [-\infty,\infty]$, we say that it is \textit{proper} if $f(\bm{x}) > -\infty$ for all $\bm{x}\in\mathbb{R}^{n}$ and its effective domain ${\rm dom}\,f:=\{\bm{x} \in \mathbb{R}^{n} : f(\bm{x})<\infty\}$ is nonempty. A proper function $f$ is said to be closed if it is lower semicontinuous. For a proper closed convex function $f: \mathbb{R}^{n} \rightarrow (-\infty,\infty]$, its subdifferential at $\bm{x}\in{\rm dom}\,f$ (which is a multifunction) is defined by $\partial f(\bm{x}):=\big\{\bm{d}\in\mathbb{R}^n: f(\bm{y}) \geq f(\bm{x}) + \langle \bm{d}, \,\bm{y}-\bm{x}\rangle, ~\forall\,\bm{y}\in\mathbb{R}^n\big\}$. % and its convex conjugate function $f^*: \mathbb{R}^{n} \rightarrow (-\infty,\infty]$ is defined by $f^*(\bm{y}):=\sup\big\{\langle \bm{y},\,\bm{x}\rangle-f(\bm{x}) : \bm{x}\in\mathbb{R}^n\big\}$. For any $\nu>0$, the Moreau envelope of $\nu f$ at $\bm{x}$ is defined by $\mathtt{M}_{\nu f}(\bm{x}) := \min_{\bm{y}} \big\{f(\bm{y}) + \frac{1}{2\nu}\|\bm{y} - \bm{x}\|^2\big\}$, and the proximal mapping of $\nu f$ at $\bm{x}$ is defined by $\mathtt{prox}_{\nu f}(\bm{x}) := \arg\min_{\bm{y}} \big\{f(\bm{y}) + \frac{1}{2\nu}\|\bm{y} - \bm{x}\|^2\big\}$. Analogously, if $f:\mathbb{R}^{n} \rightarrow [-\infty,\infty]$ is a proper closed (i.e., upper semicontinuous) concave function, its concave conjugate function at the point $\bm{y}$ can be computed as $\inf\{ f(\bm{x}) - \langle \bm{y}, \bm{x}\rangle \;:\; \bm{x}\in \mathbb{R}^n\}$.
%Let $\mathcal{C}$ be a closed convex subset of $\mathbb{R}^n$. Its indicator function $\delta_{\mathcal{C}}$ is defined by $\delta_{\mathcal{C}}(\bm{x})=0$ if $\bm{x}\in\mathcal{C}$ and $\delta_{\mathcal{C}}(\bm{x})=+\infty$ otherwise. Moreover,
Let $\mathcal{C}$ be a closed convex subset of $\mathbb{R}^n$. For any $\bm{x}\in\mathbb{R}^n$ and any symmetric positive definite matrix $H$, we define the weighted distance from $\bm{x}$ to $\mathcal{C}$ as $\mathrm{dist}_{H}(\bm{x},\mathcal{C}):=\inf_{\bm{y}\in\mathcal{C}}\|\bm{x}-\bm{y}\|_H$, and the corresponding weighted projection of $\bm{x}$ onto $\mathcal{C}$ as $\Pi_{\mathcal{C},H}(\bm{x}):=\arg\min_{\bm{y}\in\mathcal{C}}\|\bm{x}-\bm{y}\|_H$. When $H$ is the identity matrix, we omit $H$ from the notation and simply write $\mathrm{dist}(\bm{x},\mathcal{C})$ and $\Pi_{\mathcal{C}}(\bm{x})$ for the Euclidean distance to $\mathcal{C}$ and the Euclidean projection onto $\mathcal{C}$, respectively.

%\textbf{Lagrange duality.} For a problem of minimizing a proper closed convex $f$ on $\mathbb{R}^n$, a dualizing parameterization is a representation $f(\cdot) = F(\cdot, \bm{0})$ in terms of a proper function $F:\mathbb{R}^n\times \mathbb{R}^m \to (-\infty,\infty]$ such that $F(\bm{x},\bm{u})$ is closed and convex in $\bm{u}\in\mathbb{R}^m$. The associated Lagrangian $\ell:\mathbb{R}^n\times \mathbb{R}^m \to [-\infty,\infty]$ is defined by
%\[
%    \ell(\bm{x},\bm{y}):= \sup_{\bm{u}}\left\{ f(\bm{x}, \bm{u}) - \langle \bm{u}, \bm{y}\rangle\right\},
%\]
%which is convex in $\bm{x}$ and concave in $\bm{y}$, respectively. Similarly, define the function $G:\mathbb{R}^n\times \mathbb{R}^m\to [-\infty,\infty)$ as
%\[
%    G(\bm{p},\bm{y}):= \inf_{\bm{x}\in \mathbb{R}^n, \bm{u}\in \bm{y}}\left\{F(\bm{x},\bm{u}) - \langle \bm{x}, \bm{p}\rangle - \langle \bm{u}, \bm{y}\rangle\right\},
%\]
%which is a closed concave function. Then, the dual problem of $\min_{\bm{x}\in \mathbb{R}^n}\; f(\bm{x})$ is given as
%\[
%    \max_{\bm{y}\in\mathbb{R}^m}\; g(\bm{y}):= G(\bm{0}, \bm{y}).
%\]

%We refer the interested reader to the classical monographs \cite{r1970convex,rw1998variational} for comprehensive accounts of convex analysis and convex optimization.

%%%%%%%%%%%%%%%%%%%%%%%%%%%%%%%%%%%%%%%%%%%%%%%%%%%%%%%%%%%%%%%%%%%%%%%
\section{A relative-type inexact preconditioned proximal ALM} \label{sec:ripalm}

In this section, we present a relative-type inexact preconditioned proximal ALM (rip$^2$ALM)
% , originally proposed in our earlier work \cite{zlyt2024ripalm} for linearly constrained convex optimization,
for solving the general convex nonlinear programming problem \eqref{gen-mainprob}. To this end, we first follow \cite[Section 2]{es2013practical} to review the parametric convex duality framework from \cite{r1970convex,r1974conjugate} and \cite[Chapter 11]{rw1998variational}. We identify problem \eqref{gen-mainprob} with the following problem
\begin{equation}\label{gen-paraobjpri}
\min_{\bm{x}\in\mathbb{R}^{N}} \quad F(\bm{x},\bm{0}),
\end{equation}
where $F:\mathbb{R}^{N}\times\mathbb{R}^{M}\to(-\infty,+\infty]$ is defined by
\begin{equation}\label{gen-para-F}
F(\bm{x},\bm{u}) =
\left\{\begin{aligned}
& f(\bm{x}), & & \text{if}~~A\bm{x}-\bm{b}+\bm{r}=\bm{0} ~~\text{and}~~g(\bm{x}) + \bm{s} \leq \bm{0}, \\
& +\infty, & & \text{otherwise},
\end{aligned}\right.
\end{equation}
with $\bm{u}:=(\bm{r},\bm{s})\in\mathbb{R}^M$, $\bm{r}\in\mathbb{R}^{m_1}$, $\bm{s}\in\mathbb{R}^{m_2}$, and $M = m_1 + m_2$. Here, the second argument of $F$ represents some kind of perturbation to the primal problem \eqref{gen-mainprob}. %Note that $F:\mathbb{R}^{N}\times\mathbb{R}^{M}\to(-\infty,+\infty]$ is proper closed convex since both $f$ and $g$ are proper closed convex.
Then, the (ordinary) Lagrangian function of problem \eqref{gen-mainprob} can be defined by taking the concave conjugate of $F$ with respect to its second argument, that is,
\begin{equation}\label{gen-para-lag}
\begin{aligned}
\ell(\bm{x}, (\bm{\lambda},\bm{\mu}))
:=& \inf_{\bm{u} \in \mathbb{R}^{M}}\big\{F(\bm{x}, \bm{u})-\langle \bm{u}, \,(\bm{\lambda},\bm{\mu})\rangle\big\}   \\
=&\begin{cases} f(\bm{x})+\langle\bm{\lambda}, \,A\bm{x}-\bm{b}\rangle + \langle\bm{\mu}, \,g(\bm{x})\rangle, & \bm{\mu} \geq \bm{0}, \\
-\infty, & \text{otherwise},
\end{cases}
\end{aligned}
\end{equation}
where $\bm{\lambda} \in \mathbb{R}^{m_1}$ and $\bm{\mu} \in \mathbb{R}^{m_2}$ are Lagrange multipliers. Clearly, $\ell: \mathbb{R}^{N} \times \mathbb{R}^{M} \rightarrow[-\infty, \infty]$ is convex in its first argument and concave in the second argument.

For a given penalty parameter $\sigma > 0$, the augmented Lagrangian function of problem \eqref{gen-mainprob} is defined as follows (see \cite[Example 11.57]{rw1998variational}): for any $(\bm{x},\bm{\lambda},\bm{\mu})\in\mathbb{R}^{N}\times\mathbb{R}^{m_{1}}\times\mathbb{R}_{+}^{m_{2}}$,
\begin{equation}\label{gen-auglag}
\begin{aligned}
& \mathcal{L}_{\sigma}(\bm{x},\bm{\lambda},\bm{\mu}) \\
:= & \;\sup\limits_{\bm{\xi}_{1}\in\mathbb{R}^{m_1},\,\bm{\xi}_{2}\in\mathbb{R}^{m_2}} ~ \left\{ \ell(\bm{x},(\bm{\xi}_1,\bm{\xi}_2)) - \frac{1}{2\sigma}\|\bm{\lambda}-\bm{\xi}_1\|^2 - \frac{1}{2\sigma}\|\bm{\mu}-\bm{\xi}_2\|^2\right\}\\
= &\; f(\bm{x}) + \langle\bm{\lambda}, A\bm{x}-\bm{b}\rangle + \frac{\sigma}{2}\|A\bm{x}-\bm{b}\|^{2}
+ \frac{1}{2\sigma}\big\|\max\{\bm{0},\bm{\mu}+\sigma g(\bm{x})\}\big\|^{2}-\frac{1}{2\sigma}\|\bm{\mu}\|^{2}.
\end{aligned}
\end{equation}

With the above preparations, we are now ready to present rip$^2$ALM for solving problem \eqref{gen-mainprob} in Algorithm \ref{algo:rip2ALM}. We next highlight several potential advantages of rip$^2$ALM from both theoretical and numerical perspectives in the following paragraphs.

\begin{algorithm}[ht]
\caption{The rip$^2$ALM for solving problem \eqref{gen-mainprob}}\label{algo:rip2ALM}
\textbf{Input:} $\rho\in[0,1)$, $\{\sigma_{k}\}_{k=0}^{\infty}\subset\mathbb{R}_{++}$ with $\inf_{k\geq0}\{\sigma_{k}\}>0$, $\{\tau_{k}\}_{k=0}^{\infty}\subset\mathbb{R}_{++}$ with $\inf_{k\geq0}\{\tau_{k}\}>0$,
and $S\in\mathbb{R}^{N\times N}$ is a symmetric positive definite matrix. Choose $\bm{w}^0, \,\bm{x}^0\in\mathbb{R}^N$, $\bm{\lambda}^0\in\mathbb{R}^{m_1}$, $\bm{\mu}^0\in\mathbb{R}_{+}^{m_2}$ arbitrarily. Set $k=0$.  \vspace{1mm} \\
\textbf{while} the termination criterion is not met, \textbf{do} \vspace{-2mm}
\begin{itemize}[leftmargin=2.2cm]
\item[\textbf{Step 1}.] Approximately solve the subproblem
		\begin{equation}\label{rip2ALM-subpro}
			\min\limits_{\bm{x}\in\mathbb{R}^{N}}~~
			\mathcal{L}_{\sigma_{k}}(\bm{x},\,\bm{\lambda}^k,\,\bm{\mu}^k)
			+ \frac{\tau_k}{2\sigma_{k}}\big\|\bm{x}-\bm{x}^k\big\|^2_{S}
		\end{equation}
		to find a point $\bm{x}^{k+1}$ and its associated error $\Delta^{k+1}$
		such that
		\begin{equation}\label{rip2ALM-inexcond}
			\Delta^{k+1}\in
			\partial_{x}\mathcal{L}_{\sigma_{k}}(\bm{x}^{k+1},\,\bm{\lambda}^k,\,\bm{\mu}^k) + \tau_k\sigma_{k}^{-1}S\big(\bm{x}^{k+1}-\bm{x}^{k}\big),
		\end{equation}
		satisfying the following relative-type error criterion
		\begin{equation}\label{rip2ALM-stopcrit}
			\hspace{-6mm}
			\begin{aligned}
				&\; 2\big|\langle\bm{w}^k-\bm{x}^{k+1},
				\,\sigma_{k}\Delta^{k+1}\rangle\big|
				+ \big\|\sigma_{k}\Delta^{k+1}\big\|^2 \\[3pt]
				\leq&\; \rho\left(
				\big\|\sigma_{k}\big(A\bm{x}^{k+1}-\bm{b}\big)\big\|^2
				+ \big\|\min\big\{\bm{\mu}^{k},-\sigma_{k}g(\bm{x}^{k+1})\big\}\big\|^2
				+ \tau_k\big\|\bm{x}^{k+1}-\bm{x}^k\big\|^2_{S}\right).
			\end{aligned}
		\end{equation}
		
\item[\textbf{Step 2}.] Update
		\begin{equation}\label{rip2ALM_ywupdate}
			\begin{aligned}
				\bm{\lambda}^{k+1} &= \bm{\lambda}^k +
				\sigma_{k}\big(A\bm{x}^{k+1}-\bm{b}\big), \\
				\bm{\mu}^{k+1} &= \max\big\{\bm{0},\,\bm{\mu}^{k}+\sigma_{k}g(\bm{x}^{k+1})\big\}, \\
				\bm{w}^{k+1} &= \bm{w}^k - \sigma_{k}\Delta^{k+1}.
			\end{aligned}
		\end{equation}
		
\item[\textbf{Step 3}.] Set $k=k+1$ and go to \textbf{Step 1}.  \vspace{-2mm}
\end{itemize}
\textbf{end while}  \\
\textbf{Output:} $(\bm{x}^k,\bm{\lambda}^k,\bm{\mu}^k)$ \vspace{0.5mm}
\end{algorithm}

First, unlike the classical ALM, whose subproblem objective is typically only convex (and hence requires restrictive conditions to ensure strong convexity), the inclusion of the preconditioned proximal term $\frac{\tau_k}{2\sigma_k}\| \bm{x}-\bm{x}^k\|^2_{S}$ renders the objective function in \eqref{rip2ALM-subpro} strongly convex. This guarantees the uniqueness of each subproblem's solution. Moreover, the strong convexity facilitates the direct application of efficient algorithms, such as (accelerated) proximal gradient methods \cite{beck2009fast,nesterov1983method} and semismooth Newton methods \cite{qi1993nonsmooth}, to solve the subproblems with both strong theoretical guarantees and robust practical performance, without imposing additional assumptions that are often unverifiable in practice.

Second, another advantage of rip$^2$ALM lies in its convergence guarantees (see the next section), which are achieved \textit{without} introducing any additional correction step. This contrasts with the relative-type inexact variant of ALM developed in \cite{ylct2024corrected}, where an extra correction step was incorporated to establish a link between the resulting algorithm and a relative-type inexact proximal point algorithm applied to the primal-dual solution mapping. Although such a design ensures theoretical convergence, numerical results in our earlier work \cite{zlyt2024ripalm} show that this correction step not only incurs additional computational overhead, but may also degrade the practical performance of the proximal ALM (pALM) and, in some cases, even induce numerical instability. By carefully designing the relative-type error criterion \eqref{rip2ALM-stopcrit} to operate directly within the vanilla preconditioned pALM framework, the proposed rip$^2$ALM in Algorithm \ref{algo:rip2ALM} avoids the need for such a correction step, resulting in a simpler algorithmic framework that exhibits enhanced robustness and improved efficiency in practice.

When compared with the absolute-type inexact pALM studied in \cite{lst2020asymptotically,r1976augmented}, the proposed rip$^2$ALM is more user-friendly in terms of tolerance parameter tuning. The absolute-type framework requires pre-specifying a summable sequence of infinitely many tolerance parameters to control the accuracy of subproblems 
minimization, which may entail considerable effort for selecting an
appropriate sequence to avoid excessive cost in solving the subproblems. 
%%and heavily relies on prior experience. 
In contrast, rip$^2$ALM involves only a \textit{single} tolerance parameter $\rho\in[0,1)$ in \eqref{rip2ALM-stopcrit}, which can be efficiently selected via a simple one-dimensional grid search. This greatly enhances its ease of implementation and reliability in practice. Finally, we refer the reader to our earlier work \cite[Section 5.1]{zlyt2024ripalm} for detailed numerical comparisons among different inexact pALM variants.

We conclude this section by emphasizing the role of the preconditioner $S$. When $S=I$, the resulting method reduces to the classical proximal ALM framework with a relative-type error criterion. In large-scale implementations, however, it is often advantageous to construct $S$ by exploiting problem structure so that it can improve the conditioning and cost of solving the inner subproblem. As an illustration, consider the following standard dual linear program problem:
\begin{equation*}
\min_{\bm{y}\in\mathbb{R}^m}\ \bm{b}^\top\bm{y}
\quad \mathrm{s.t.}\quad \bm{c} - A^\top\bm{y} \ge \bm{0},
\end{equation*}
where $\bm{b}\in \mathbb{R}^m$, $\bm{c}\in\mathbb{R}^n$, and $A\in \mathbb{R}^{m\times n}$ has full row rank. One can derive the (reduced) augmented Lagrangian as
\begin{equation*}
{\mathcal{L}}_\sigma(\bm{y};\bm{x})
= \bm{b}^\top\bm{y}
+ \frac{\sigma}{2}\left\|\Pi_+\left(A^{\top}\bm{y}-\bm{c} 
+ \bm{x}/\sigma\right)\right\|^2
- \frac{1}{2\sigma}\|\bm{x}\|^2,
\end{equation*}
where $\sigma>0$ is the penalty parameter, $\bm x$ is the primal variable, and $\Pi_+(\cdot)$ denotes the projection operator onto $\mathbb{R}^n_+$. Accordingly, a proximal ALM step in Algorithm \ref{algo:rip2ALM} computes $\bm y^{k+1}$ by (approximately) minimizing the function $\phi_k(\bm{y}):={\mathcal{L}}_{\sigma_k}(\bm{y};\bm{x}^k)+\tfrac{\tau_k}{2\sigma_k}\|\bm{y}-\bm{y}^k\|_{S}^2$.

A highly efficient method for minimizing the function $\phi_k$  is the semismooth Newton method, which exploits the strong semismoothness of the projection operator $\Pi_+$; interested readers are referred to \cite{lst2020asymptotically} for its excellent convergence properties and practical performance. Each semismooth Newton step requires solving an $m\times m$ symmetric positive definite linear system of the form
\begin{equation*}
\left(\tfrac{\tau_k}{\sigma_k}S + \sigma_k A D_k A^\top\right)\bm{d} 
= -\nabla \phi_k(\bm{y}),
\end{equation*}
where $D_k$ is a diagonal matrix with entries in $\{0,1\}$ determined by the active set of the projection. For large-scale problems, forming or factorizing $\left(\tfrac{\tau_k}{\sigma_k}S + \sigma_k A D_k A^\top\right)\in \mathbb{S}_{++}^m$ repeatedly can be expensive. To make the algorithm scalable and robust, a natural structure-driven choice is $S=AA^\top$ (which is positive definite because the matrix $A$ has full row rank). Using the preconditioned conjugate gradient (PCG) method with $S$ as the preconditioner for solving the above linear system avoids the explicit formation of $A D_k A^\top$ and yields a better uniformly controlled conditioning. Indeed, one can verify that
\begin{equation*}
\mathrm{cond}\!\left((AA^\top)^{-1}\Big(\tfrac{\tau_k}{\sigma_k}AA^\top+\sigma_k A D_k A^\top\Big)\right)
= \mathrm{cond}\!\left(\tfrac{\tau_k}{\sigma_k}I+\sigma_k\,(AA^\top)^{-1/2}A D_k A^\top (AA^\top)^{-1/2}\right).
\end{equation*}
Since $0\preceq D_k\preceq I$ implies that $\lambda_{\max}((AA^\top)^{-1/2}A D_k A^\top (AA^\top)^{-1/2})\le 1$, one obtains that 
\begin{equation*}
\mathrm{cond}\!\left((AA^\top)^{-1}\Big(\tfrac{\tau_k}{\sigma_k}AA^\top+\sigma_k A D_k A^\top\Big)\right) \leq  1+\frac{\sigma_k^2}{\tau_k}. 
%\leq 1 + \frac{(\sup_{k\geq 0}\{\sigma_k\})^2}{\inf_{k\geq 0}\{\tau_k\}}.
\end{equation*}
In the above, $\mathrm{cond}(H)$ denotes the condition number of a symmetric positive definite matrix $H$, defined as the ratio of the largest and smallest eigenvalues of $H$. On the other hand, one can verify that 
\begin{equation*}
\mathrm{cond}\!\left(\tfrac{\tau_k}{\sigma_k}I+\sigma_k A D_k A^\top\right) \leq \frac{\frac{\tau_k}{\sigma_k} + \sigma_k\lambda_{\rm max}(AA^\top)}{\frac{\tau_k}{\sigma_k}} = 1 + \frac{\sigma_k^2}{\tau_k} \lambda_{\rm max}(AA^\top),
\end{equation*}
which can be much larger than the previous bound when $\lambda_{\rm max}(AA^\top)$ is large. This example highlights that while $S=I$ provides a clean baseline, a carefully chosen preconditioner such as $S=AA^\top$ can potentially reduce the conditioning and hence the cost of the inner Newton/PCG solves, thereby improving the practical performance of proximal ALM in the large-scale setting.

%%%%%%%%%%%%%%%%%%%%%%%%%%%%%%%%%%%%%%%%%%%%%%%%%%%%%%
\section{Convergence analysis}\label{section:analysis}

In this section, we conduct a comprehensive analysis of the convergence properties of rip$^2$ALM in Algorithm~\ref{algo:rip2ALM}. To set the stage, we begin by recalling necessary definitions and preliminaries (more details can be found in \cite[Section 2]{es2013practical}). First, recall the definition of the function $F$ from \eqref{gen-para-F}, and define its concave conjugate $G:\mathbb{R}^N \times \mathbb{R}^M \to [-\infty,+\infty)$. Specifically, the function $G$ is defined as
\begin{equation*}
G(\bm{p},\bm{y}) := \inf_{\bm{x} \in \mathbb{R}^N,\, \bm{u} \in \mathbb{R}^M}
\Big\{F(\bm{x}, \bm{u}) - \langle \bm{x}, \bm{p}\rangle - \langle \bm{u}, \bm{y}\rangle\Big\}.
\end{equation*}
By definition, $G$ is a closed (i.e., upper semicontinuous) concave function.
Then, the dual problem of \eqref{gen-paraobjpri} is given by
\begin{equation}\label{gen-paraobjdual}
\max_{\bm{y} \in \mathbb{R}^M} \; G(\bm{0}, \bm{y}).
\end{equation}

\medskip
\noindent\textbf{Subdifferentials.}
In order to establish a precise connection between the primal problem \eqref{gen-paraobjpri}, dual problem \eqref{gen-paraobjdual}, and Lagrangian function \eqref{gen-para-lag}, we introduce the subdifferential mappings of the relevant functions. Let $\partial F: \mathbb{R}^N \times \mathbb{R}^M \rightrightarrows \mathbb{R}^N \times \mathbb{R}^M$ and $\partial G: \mathbb{R}^N \times \mathbb{R}^M \rightrightarrows \mathbb{R}^N \times \mathbb{R}^M$ denote the subdifferentials of $F$ and $G$, respectively. Concretely, these mappings are defined by
\begin{align*}
(\bm{p}, \bm{y}) \in \partial F(\bm{x},\bm{u})
&\;\;\Leftrightarrow\;\; F(\bm{x}', \bm{u}') \geq F(\bm{x}, \bm{u})
+ \langle \bm{p}, \bm{x}' - \bm{x}\rangle
+ \langle \bm{y}, \bm{u}' - \bm{u}\rangle,
\quad \forall\,(\bm{x}',\bm{u}'), \\[0.3em]
(\bm{x}, \bm{u}) \in \partial G(\bm{p},\bm{y})
&\;\;\Leftrightarrow\;\; G(\bm{p}', \bm{y}') \leq G(\bm{p}, \bm{y})
- \langle \bm{x}, \bm{p}' - \bm{p}\rangle
- \langle \bm{u}, \bm{y}' - \bm{y}\rangle,
\quad \forall\,(\bm{p}',\bm{y}').
\end{align*}
Next, recalling the definition of the Lagrangian function $\ell$ from \eqref{gen-para-lag}, we introduce its subdifferential mapping $\partial \ell: \mathbb{R}^N \times \mathbb{R}^M \rightrightarrows \mathbb{R}^N \times \mathbb{R}^M$, defined by
\begin{equation*}
(\bm{p}, \bm{u}) \in \partial \ell(\bm{x}, \bm{y}) \;\;\Leftrightarrow\;\;
\begin{cases}
\ell(\bm{x}', \bm{y}) \geq \ell(\bm{x}, \bm{y}) + \langle \bm{p}, \bm{x}'-\bm{x}\rangle, & \forall\,\bm{x}' \in \mathbb{R}^N, \\[0.3em]
\ell(\bm{x}, \bm{y}') \leq \ell(\bm{x}, \bm{y}) - \langle \bm{u}, \bm{y}'-\bm{y}\rangle, & \forall\,\bm{y}' \in \mathbb{R}^M.
\end{cases}
\end{equation*}
For clarity in our context, we partition the dual vector $\bm{y}$ as $\bm{y}=(\bm{\lambda}, \bm{\mu}) \in \mathbb{R}^M$, where $\bm{\lambda} \in \mathbb{R}^{m_1}$ and $\bm{\mu} \in \mathbb{R}^{m_2}$. Then, when $\bm{\mu}\geq 0$, the above definition yields
\begin{equation}\label{Lagform}
\partial \ell(\bm{x}, (\bm{\lambda}, \bm{\mu})) =
\big\{ \partial f(\bm{x}) + A^{\top}\bm{\lambda} + \nabla g(\bm{x})^{\top}\bm{\mu} \big\}
\times\Big( \{\bm{b}-A\bm{x}\}
\times \big\{-g(\bm{x}) + \mathcal{N}_{\mathbb{R}_+^{m_2}}(\bm{\mu})\big\}\Big),
\end{equation}
where $\nabla g(\bm{x})$ is the Jacobian of $g$ at $\bm{x}$, and $\mathcal{N}_{\mathbb{R}_+^{m_2}}$ denotes the normal cone to the nonnegative orthant $\mathbb{R}_+^{m_2}$.

\medskip
\noindent\textbf{Relationship between relevant subdifferentials.}
It is clear that the three set-valued mappings $\partial F$, $\partial G$, and $\partial \ell$ are all maximal monotone operators.
%\footnote{Let $\mathcal{H}$ be a Hilbert space. A set-valued mapping $T:\mathcal{H}\rightrightarrows\mathcal{H}$ is said to be monotone if $\langle T(\bm{x}) - T(\bm{y}), \,\bm{x}-\bm{y}\rangle \geq 0$ holds for any $\bm{x},\bm{y}\in\mathcal{H}$. Moreover, $T$ is said to be maximally monotone if the graph of $T$, denoted as $\mathrm{Gph}(T):=\{(\bm{x},\bm{y})\in\mathcal{H}\times \mathcal{H}\mid \bm{y}\in T(\bm{x})\}$, is not properly contained in the graph of any other monotone mapping \cite{r1976monotone}.}
Moreover, they are closely related through the equivalence (see also \cite[Equation~(23)]{es2013practical})
\begin{equation}\label{gen-subdiffrelation}
(\bm{p}, \bm{y}) \in \partial F(\bm{x}, \bm{u})
\;\;\Leftrightarrow\;\;
(\bm{p}, \bm{u}) \in \partial \ell(\bm{x}, \bm{y})
\;\;\Leftrightarrow\;\;
(\bm{x}, \bm{u}) \in \partial G(\bm{p}, \bm{y}).
\end{equation}
Intuitively, $\partial F$ and $\partial G$ are inverses of one another, while $\partial \ell$ can be viewed as a partial inverse of both.

\medskip
\noindent\textbf{Saddle points and strong duality.}
We now arrive at a fundamental conclusion, which is the optimality condition for the primal and dual problems: if $(\bm{x}^*, \bm{y}^*) \in \mathbb{R}^N \times \mathbb{R}^M$ satisfies
\begin{equation*}
(\bm{0}, \bm{0}) \in \partial \ell(\bm{x}^*, \bm{y}^*),
\end{equation*}
then $\bm{x}^*$ solves the primal problem \eqref{gen-paraobjpri} and $\bm{y}^*$ solves the dual problem \eqref{gen-paraobjdual}. In this case, we call $(\bm{x}^*, \bm{y}^*)$ a \emph{saddle point} of the Lagrangian function $\ell$. The existence of a saddle point immediately implies \emph{strong duality}, namely,
\begin{equation*}
F(\bm{x}^*, \bm{0}) = G(\bm{0}, \bm{y}^*),
\end{equation*}
so that the optimal values of the primal problem \eqref{gen-paraobjpri} and the dual problem \eqref{gen-paraobjdual} are well-defined and coincide, i.e., there is no duality gap. In view of this, the set of saddle points can be expressed in product form as $\mathcal{X}^* \times \mathcal{Y}^* \subset \mathbb{R}^N \times \mathbb{R}^M$, where $\mathcal{X}^*$ denotes the solution set of the primal problem \eqref{gen-paraobjpri} (equivalently, problem \eqref{gen-mainprob}), and $\mathcal{Y}^*$ denotes the solution set of the dual problem \eqref{gen-paraobjdual}.

%%%%%%%%%%%%%%%%%%%%%%%%%%%%%%%%%%%%%%%%
\subsection{Global convergence}

With the above preparations, we are now ready to establish the convergence of the proposed rip$^2$ALM in Algorithm \ref{algo:rip2ALM}.

\begin{theorem}\label{thm:convergence-gen}
Let the functions $F$, $G$ and $\ell$ be defined as in \eqref{gen-paraobjpri}, \eqref{gen-paraobjdual} and \eqref{gen-para-lag}, respectively. Let $\rho\in[0,1)$, $S\in\mathbb{R}^{N\times N}$ be a symmetric positive definite matrix, $\{\sigma_{k}\}$ be a positive sequence satisfying that $\sigma_k\geq\sigma_{\min}>0$ for all $k\geq0$, and $\{\tau_k\}$ be a positive sequence satisfying that
\begin{equation*}
\tau_{k}\geq\tau_{\min}>0, \quad \tau_{k+1}\leq(1+\nu_{k})\tau_{k} \quad
\mbox{with} \quad \nu_{k}\geq0 ~~\mbox{and}~~ {\textstyle\sum_{k=0}^{\infty}}\nu_{k} < +\infty.
\end{equation*}
Let $\{\bm{x}^{k}\}$, $\{\Delta^{k}\}$, $\{\bm{w}^{k}\}\subset\mathbb{R}^{N}$ and $\{\bm{y}^{k}:=(\bm{\lambda}^k,\bm{\mu}^k)\}\subset\mathbb{R}^{M}$ be sequences generated by the rip$^2$ALM in Algorithm \ref{algo:rip2ALM}. If $\ell$ admits a saddle point (i.e., $(\partial\ell)^{-1}(\bm{0},\bm{0})\neq\emptyset$), then the following statements hold.
\begin{enumerate}[label=(\roman*), left=-5pt]
\item The sequences $\{\bm{x}^k\}$, $\{\bm{w}^k\}$ and $\{\bm{y}^k\}$ are all bounded.
		
\item $\lim\limits_{{k}\to\infty}\Delta^{k+1}=\bm{0}$, $\lim\limits_{{k}\to\infty}\bm{p}^{k+1}=\bm{0}$ and $\lim\limits_{{k}\to\infty}\bm{u}^{k+1}=\bm{0}$, where $\bm{p}^{k+1}$ and $\bm{u}^{k+1}$ are defined by
\begin{equation*}
\bm{p}^{k+1}:=\Delta^{k+1}
- \tau_k\sigma_{k}^{-1}S(\bm{x}^{k+1}
- \bm{x}^{k})\quad\text{and}\quad
\bm{u}^{k+1}:=\sigma_{k}^{-1}(\bm{y}^{k}-\bm{y}^{k+1}), \quad \forall\,k\geq0.
\end{equation*}
		
\item Both the sequences $\left\{F(\bm{x}^{k+1},\,\bm{u}^{k+1})\right\}$ and $\left\{G(\bm{p}^{k+1},\bm{y}^{k+1})\right\}$ converge to the common optimal value of problems \eqref{gen-paraobjpri} and \eqref{gen-paraobjdual}.
		
\item Any accumulation point of $\{\bm{x}^k\}$ is an optimal solution of primal problem \eqref{gen-paraobjpri} (i.e., problem \eqref{gen-mainprob}), and any accumulation point of $\{\bm{y}^k\}$ is an optimal solution of dual problem \eqref{gen-paraobjdual}.
		
\item The sequence $\{\bm{y}^k\}$ converges to an optimal solution of dual problem \eqref{gen-paraobjdual}.
\end{enumerate}
\end{theorem}
\begin{proof}
See Appendix \ref{sec:proof-conver-gen}.
\end{proof}

Note that the conditions on the proximal parameter sequence $\{\tau_k\}$ are quite mild and easy to satisfy. For example, one may simply set $\tau_k \equiv \tau > 0$ for all $k \geq 0$, in which case $\nu_k\equiv0$. Moreover, under the assumption that $\ell$ admits at least one saddle point, which is itself a rather weak requirement in convex optimization (which can be ensured by assuming a certain constraint qualification condition, e.g., MFCQ \cite{mangasarian1967fritz}), both the sequences $\{\bm{x}^k\}$ and $\{\bm{y}^k\}$ are bounded. Consequently, each sequence has at least one accumulation point, and by Theorem \ref{thm:convergence-gen}, any such accumulation point is an optimal solution of the respective primal or dual problem. %In Theorem~\ref{thm:convergence-gen}(v), we have established the convergence of $\{\bm{y}^k\}$ for the rip$^2$ALM in Algorithm \ref{algo:rip2ALM}.
In contrast, for the relative-type inexact ALM \textit{without} a proximal term proposed by Eckstein and Silva \cite{es2013practical}, one can only guarantee the boundedness of $\{\bm{y}^k\}$. %This highlights the potential theoretical advantage of introducing a proximal term in our ripALM.

Finally, we would like to point out that, unlike the absolute-type inexact proximal ALM \cite{lst2020asymptotically,r1976augmented}, it remains unclear whether our proposed rip$^2$ALM in Algorithm \ref{algo:rip2ALM} can be interpreted as an application of the inexact (preconditioned) proximal point algorithm (PPA) to the associated primal-dual solution mapping. Therefore, at this stage, we can only establish the global sequential convergence of the dual sequence $\{\bm{y}^k\}$ in Theorem \ref{thm:convergence-gen}, via a direct proof, which is inspired by \cite{as2016note,es2013practical} but is substantially more involved due to the presence of the preconditioned proximal term $\frac{\tau_{k}}{2\sigma_{k}}\big\|\bm{x}-\bm{x}^k\big\|^2_S$. In particular, we need to develop a new recursive inequality \eqref{gen-recursion_of_xyw} that simultaneously involves both primal and dual sequences. Moreover, as we shall see later that, under an additional error bound condition, we can also prove the global sequential convergence of the primal sequence $\{\bm{x}^k\}$, which fails to hold for Eckstein and Silva's relative-type inexact ALM \cite{es2013practical}. This further highlights another theoretical advantage of introducing a (preconditioned) proximal term in rip$^2$ALM.

%%%%%%%%%%%%%%%%%%%%%%%%%%%%%%%%%%%%%%%%%%%%%%%%%%%%%%%%%%%%
\subsection{Asymptotic (super)linear convergence rate}

In this subsection, we establish the asymptotic (super)linear convergence rate of rip$^2$ALM under an error bound condition presented in Assumption \ref{asp:error-bound-Li-gen}.

\begin{assumption}\label{asp:error-bound-Li-gen}
For any $r>0$, there exists a constant $\kappa>0$ such that, for any $(\bm{x},\bm{y}) \in \left\{(\bm{x},\bm{y})\in\mathbb{R}^N\times\mathbb{R}^M \mid \mathrm{dist}\left((\bm{x},\bm{y}), \,(\partial \ell)^{-1}(\bm{0}, \bm{0})\right) \leq r\right\}$,
\begin{equation}\label{eq:error-bound-Li}
\mathrm{dist}\left((\bm{x},\bm{y}), \,(\partial\ell)^{-1}(\bm{0}, \bm{0})\right)
\leq \kappa\,\mathrm{dist}\left((\bm{0}, \bm{0}), \,\partial\ell(\bm{x},\bm{y})\right).
\end{equation}
\end{assumption}

Imposing a suitable error bound condition is a standard approach for deriving the fast asymptotic convergence rates of PPA-type and ALM-type algorithms in the convex setting; see, e.g., \cite{lst2018highly,lst2020asymptotically,l1984asymptotic,r1976augmented,r1976monotone,ylct2024corrected,zc2020linear}. In Rockafellar's seminal works \cite{r1976augmented,r1976monotone}, this was achieved under the assumption that $(\partial\ell)^{-1}$ is Lipschitz continuous at the origin with modulus $\kappa \geq 0$. Specifically, this assumption requires that there exists a \textit{unique} solution $(\bm{x}^*,\bm{y}^*)$ such that $(\bm{0},\bm{0}) \in \partial\ell(\bm{x}^*,\bm{y}^*)$ (i.e., $(\partial\ell)^{-1}(\bm{0},\bm{0})=\{(\bm{x}^*,\bm{y}^*)\}$), and that for some $\kappa>0$ and $r>0$, it holds that
\begin{equation*}
\big\|(\bm{x}-\bm{x}^*, \;\bm{y}-\bm{y}^*)\big\| \leq \kappa\,\|(\bm{p}, \bm{u})\|
\quad \mathrm{whenever} \quad
(\bm{x},\bm{y})\in(\partial \ell)^{-1}(\bm{p}, \bm{u})
~~\mathrm{and}~~\|(\bm{p}, \bm{u})\| \leq r.
\end{equation*}
This condition, however, is rather restrictive as it forces the solution set to be a singleton.

To relax this restriction, Luque \cite{l1984asymptotic} introduced a growth condition for establishing the convergence rate of PPA. His condition is essentially equivalent to the concept of local \textit{upper} Lipschitz continuity, earlier proposed by Robinson \cite{r1976implicit,r1981some}. Specifically, $(\partial\ell)^{-1}$ is called locally \textit{upper} Lipschitz continuous at the origin if $(\partial\ell)^{-1}(\bm{0},\bm{0})\neq\emptyset$ and there exist $r>0$ and $\kappa>0$ such that
\begin{equation*}
\mathrm{dist}\big((\bm{x},\bm{y}), \,(\partial\ell)^{-1}(\bm{0}, \bm{0})\big)
\leq \kappa\,\|(\bm{p}, \bm{u})\|
\quad \mathrm{whenever} \quad
(\bm{x},\bm{y})\in(\partial \ell)^{-1}(\bm{p}, \bm{u})
~~\mathrm{and}~~\|(\bm{p}, \bm{u})\| \leq r.
\end{equation*}
Unlike the Lipschitz continuity condition used by Rockafellar, the above relaxed condition allows for non-unique solutions, yet it remains sufficient to guarantee the asymptotic (super)linear convergence of the iterates in terms of their distance to the solution set or saddle points; see, e.g., \cite{lst2018highly,zc2020linear}.

As noted in \cite[Lemma 2.4]{lst2020asymptotically}, the error bound condition \eqref{eq:error-bound-Li} is even weaker than the local upper Lipschitz continuity of $(\partial\ell)^{-1}$ at the origin. The latter was used in \cite{zc2020linear} to establish the asymptotic (super)linear convergence rate of Eckstein and Silva's relative-type inexact ALM, while the former, weaker condition has also been adopted in \cite{lst2020asymptotically} and \cite{ylct2024corrected} to derive similar asymptotic rate estimates.

We are now ready to present the main results regarding the asymptotic fast convergence rate of rip$^2$ALM.

\begin{theorem}\label{thm:Q-linear-rate-gen}
Let $\ell$ be defined as in \eqref{gen-para-lag}, and let $\rho\in[0,1)$, $\{\sigma_{k}\}$ be a positive sequence satisfying that $\sigma_k\geq\sigma_{\min}>0$ for all $k\geq0$, $S\in\mathbb{R}^{N\times N}$ be a symmetric positive definite matrix, and $\{\tau_{k}\}$ be a positive sequence satisfying that
\begin{equation*}
\tau_{k}\geq\tau_{\min}>0,
\quad \tau_{k+1}\leq(1+\nu_{k})\tau_{k}
\quad \mbox{with} \quad \nu_{k}\geq0
~~\mbox{and}~~
{\textstyle\sum_{k=0}^{\infty}}\nu_{k} < +\infty.
\end{equation*}
Suppose additionally that $\ell$ admits a saddle point (i.e., $(\partial\ell)^{-1}(\bm{0},\bm{0})\neq\emptyset$), Assumption \ref{asp:error-bound-Li-gen} holds, and the sequences of parameters $\rho$, $\{\sigma_k\}$ and $\{\tau_{k}\}$ satisfy that
\begin{equation}\label{para-conds}
\sqrt{\tau_{\min}\lambda_{\min}(S)} - 2\sqrt{\rho} > 0
\quad \text{and} \quad
\liminf\limits_{k\to\infty} ~ \sigma_{k} > c\cdot\frac{2\kappa\sqrt{\tau_{\max}\lambda_{\max}(S)}\left(\rho+\sqrt{\rho\,\overline{\tau}_{\max}}\right)}{\sqrt{\tau_{\min}\lambda_{\min}(S)} - 2\sqrt{\rho}},
\end{equation}
where $c>1$ is an arbitrarily given positive constant, $\lambda_{\min}(S)$ ($\lambda_{\max}(S)$) is the smallest (largest) eigenvalue of $S$, $\tau_{\max}:=\tau_{0} \prod_{k=0}^{\infty}(1+\nu_{k})$, and $\overline{\tau}_{\max}:=\max\left\{1,\tau_{\max}\lambda_{\max}(S)\right\}$. Let $\Lambda_k := \mathrm{Diag}(\tau_{k}S,I_{M})$, $\overline{\tau}_{k} := \max\left\{1, \tau_{k}\lambda_{\max}(S)\right\}$, and
\begin{equation*}
\gamma_{k} := \left(1 - \frac{2\kappa\sqrt{\tau_{k}\lambda_{\max}(S)}\left(\rho+\sqrt{\rho\overline{\tau}_{k}}\right) + 2\sigma_{k}\sqrt{\rho}}{\sigma_{k}\sqrt{\tau_{k}\lambda_{\min}(S)}}\right) \frac{\sigma_{k}^{2}}{\kappa^2\left(\sqrt{\rho}+\sqrt{\overline{\tau}_{k}}\right)^2\overline{\tau}_{k}}.
\end{equation*}
Then, the following statements hold.
\begin{enumerate}[label=(\roman*)]
\item For any sufficiently large $k$, we have that
		\begin{equation*}
			\gamma_{k}\geq
			\left(\frac{c-1}{c}\right)
			\cdot\frac{\sqrt{\tau_{\min}\lambda_{\min}(S)} - 2\sqrt{\rho}}{\sqrt{\tau_{\min}\lambda_{\min}(S)}}
			\cdot\frac{\sigma_k^2}{\kappa^2\left(\sqrt{\rho}+\sqrt{\overline{\tau}_{\max}}\right)^2\overline{\tau}_{\max}} > 0,
		\end{equation*}
		and
		\begin{equation*}%\label{eq:linear_rate}
			\mathrm{dist}_{\Lambda_{k+1}}\left((\bm{x}^{k+1},\bm{y}^{k+1}),
			\,(\partial\ell)^{-1}(\bm{0}, \bm{0})\right) \leq \mu_{k}\, \mathrm{dist}_{\Lambda_k}\left((\bm{x}^{k},\bm{y}^{k}), \,(\partial\ell)^{-1}(\bm{0}, \bm{0})\right),
		\end{equation*}
		where
		\begin{equation*}
			%\limsup\limits_{k\to\infty}\,\left\{\mu_{k}:=\sqrt{\frac{1+\nu_{k}}{1+\gamma_{k}}}\right\} < 1.
			\mu_{k}:=\sqrt{\frac{1+\nu_{k}}{1+\gamma_{k}}}
			\quad \mbox{satisfies} \quad
			\limsup\limits_{k\to\infty}\,\left\{\mu_{k}\right\}<1
			~~\mbox{as}~~\nu_{k}\to0.
		\end{equation*}
		
\item The whole sequence $\{\bm{x}^k\}$ is convergent.
\end{enumerate}
\end{theorem}
\begin{proof}
See Appendix \ref{sec:proof-Q-linear-rate-gen}.
\end{proof}

Note from Theorem \ref{thm:Q-linear-rate-gen} that achieving a fast asymptotic convergence rate requires a slightly stronger condition on $\rho$, namely $\rho < \frac{1}{4}\tau_{\min}\lambda_{\min}(S)$, with $\tau_{\min}\lambda_{\min}(S) < 4$. Moreover, by examining the expression of $\gamma_k$, we see that after a finite number of iterations, $\gamma_k$ becomes proportional to the squared penalty parameter $\sigma_k^2$, provided that $\rho$, $\{\sigma_k\}$ and $\{\tau_k\}$ satisfy the conditions in \eqref{para-conds}. Consequently, from the expression of $\mu_k$, we further see that choosing a sufficiently large $\sigma_k$ drives the convergence factor $\mu_k$ arbitrarily close to zero, thereby leading to an asymptotic convergence rate that approaches superlinear. However, from a computational perspective, excessively large values of $\sigma_k$ may induce numerical instability and are therefore not recommended in practice. Indeed, the practical performance gap between linear and superlinear convergence is often modest: choosing a moderately large $\sigma_k$ typically results in only a slight increase in the number of iterations compared to using excessively large values. Finally, although the error bound condition has become a standard analytical tool in the convergence rate analysis, it is generally difficult to verify, especially for general nonlinear programming problems. Future research may explore relaxing this requirement, identifying problem classes where the condition can be explicitly verified by exploiting structural properties.
%%or developing preprocessing schemes that reformulate the problems to satisfy it.

%%%%%%%%%%%%%%%%%%%%%%%%%%%%%%%%%%%%%%%%%%%%%%
\subsection{Global ergodic convergence rate}

In this subsection, we present a novel analysis of the global \emph{ergodic} convergence rate of the proposed rip$^2$ALM, \textit{without} assuming the error bound condition in Assumption \ref{asp:error-bound-Li-gen}. Specifically, we establish ergodic convergence rates with respect to two fundamental measures of optimality in convex constrained optimization: (1) the violation of the primal feasibility constraints, and (2) the residual error in the primal objective function value. Such ergodic rate estimates are commonly used for evaluating the efficiency of ALM-type methods (see, e.g., \cite{qclh2025convergence,xu2021iteration}). We provide rigorous characterizations of the averaged behavior of rip$^2$ALM across its iterates and, when combined with the convergence results established in previous sections, offer a more complete picture of the algorithm's overall performance.

To streamline the forthcoming analysis, we introduce a few useful constants: %that characterize the boundedness of the sequences generated by the algorithm and some useful quantities in the convergence proofs:
\begin{itemize}
\item $B_{x}$: an upper bound on the primal sequence $\{\bm{x}^{k}\}$, ensuring the iterates remain contained within a compact set;

\item $B_{y}$: an upper bound on the dual sequence $\{\bm{y}^{k}\}$, guaranteeing stability of the multiplier updates;

\item $\tau_{\max}:=\tau_{0}\prod_{k=0}^{\infty}(1+\nu_{k})$, which bounds the growth of the proximal parameters and plays a central role in controlling the conditioning of the subproblems;
	
\item $C_{0}:=\tfrac{\tau_{0}}{2}\|\bm{x}^{*}-\bm{x}^{0}\|_{S}^{2}+\tfrac{1}{2}\|\bm{y}^{0}\|^{2} + \tfrac{1}{2}\|\bm{x}^{*} - \bm{w}^{0}\|^{2} + \lambda_{\max}(S)(\|\bm{x}^*\|^2+B_x^2)\tau_{\max}\sum_{i=0}^{\infty}\nu_i$, a finite constant that summarizes the initial error and the cumulative effect of parameter perturbations;
	
\item $C_{xy}:= \frac{1}{2} \sum_{k=0}^{\infty}\left(\|\bm{y}^{k+1} - \bm{y}^{k}\|^{2} + \tau_{k}\|\bm{x}^{k+1} - \bm{x}^{k}\|_{S}^{2}\right)$, the accumulation of the successive change, and it can seen from \eqref{gen-addineq1} that $C_{xy}$ is a finite constant.
\end{itemize}

It follows directly from Theorem~\ref{thm:convergence-gen} and its proof that all the above constants are finite and well-defined under the standing assumptions. These quantities will serve as the building blocks in our subsequent analysis, enabling us to derive explicit bounds on the ergodic convergence rates of rip$^2$ALM. %Specifically, the next theorem shows that the ripALM can achieve an $\mathcal{O}(1/\sum_{i=0}^k\sigma_i)$ ergodic convergence rate for both primal feasibility violations and the residual gap in the objective function value.

\begin{theorem}\label{thm:ergo-rate}
Suppose that all the assumptions in Theorem \ref{thm:convergence-gen} hold. Let $(\bm{x}^{*},\bm{y}^{*})\in\mathcal{X}^{*}\times\mathcal{Y}^{*}$ be an arbitrary saddle point of $\ell$, and let $\{\bm{x}^{k}\}$ and $\{\bm{y}^{k}:=(\bm{\lambda}^k,\bm{\mu}^k)\}$ be sequences generated by the rip$^2$ALM in Algorithm \ref{algo:rip2ALM}. Define the ergodic primal sequence $\{\widehat{\bm{x}}^k\}$ as
\begin{equation*}
\widehat{\bm{x}}^k:=\frac{\sum_{i=0}^{k-1} \sigma_i\bm{x}^{i+1}}{\sum_{i=0}^{k-1}\sigma_{i}}.
\end{equation*}
Then, for all $k\geq 1$, it holds that
\begin{equation}\label{primfeas-vio}
\textnormal{\texttt{feas}}(\widehat{\bm{x}}^k) := \begin{Vmatrix}
A \widehat{\bm{x}}^k-\bm{b} \\
\max\left\{\bm{0},\,g(\widehat{\bm{x}}^k)\right\}
\end{Vmatrix}
\leq \Xi_{k} := \frac{2B_{y}}{\sum_{i=0}^{k-1}\sigma_{i}}
\end{equation}
and
\begin{equation}\label{primobj-gap}
-\|\bm{y}^{*}\|\cdot\Xi_{k} - \frac{\sigma_{k}}{2}\cdot\Xi_{k}^{2}
\leq f(\widehat{\bm{x}}^{k}) - f(\bm{x}^{*})
\leq \cfrac{C_{0} + {\rho}C_{xy}}{\sum_{i=0}^{k-1}\sigma_{i}}.
\end{equation}
\end{theorem}
\begin{proof}
See Appendix \ref{sec:proof-ergo-rate}.
\end{proof}

%\paragraph{Comments on ergodic convergence rates.}
From Theorem~\ref{thm:ergo-rate}, we see that under the same assumptions of Theorem~\ref{thm:convergence-gen}, both the sequences $\left\{\texttt{feas}(\widehat{\bm{x}}^k)\right\}
$ and $\left\{\left|f(\widehat{\bm{x}}^{k})-f(\bm{x}^{*})\right|\right\}$ converge at  the rate of $\mathcal{O}(1/\sum_{i=0}^k\sigma_i)$. In other words, the convergence behavior is governed by the growth rate of the cumulative penalty sequence $\sum_{i=0}^{k-1}\sigma_i$. Since the penalty parameters $\{\sigma_k\}$ can be chosen with considerable flexibility, different convergence rates can be achieved depending on this choice. In particular, we have the following results.

\begin{corollary}
Suppose that all assumptions in Theorem \ref{thm:ergo-rate} hold. Then, the following statements hold.
\begin{itemize}[leftmargin=0.8cm]
\item If $0 < \sigma_{\min} \leq \sigma_k \leq \sigma_{\max} < \infty$, then
		$\textnormal{\texttt{feas}}(\widehat{\bm{x}}^k) = \mathcal{O}\!\left(\tfrac{1}{k}\right)$ and $\left|f(\widehat{\bm{x}}^{k})-f(\bm{x}^{*})\right| = \mathcal{O}\!\left(\tfrac{1}{k}\right)$;
		
\item If $\sigma_k = k+1$, then
		$\textnormal{\texttt{feas}}(\widehat{\bm{x}}^k) = \mathcal{O}\!\left(\tfrac{1}{k^2}\right)$ and $\left|f(\widehat{\bm{x}}^{k})-f(\bm{x}^{*})\right| = \mathcal{O}\!\left(\tfrac{1}{k^2}\right)$;
		
\item If $\sigma_k = \sigma_0 c^k$ with some $c>1$, then
		$\textnormal{\texttt{feas}}(\widehat{\bm{x}}^k) = \mathcal{O}(c^{-k})$ and $\left|f(\widehat{\bm{x}}^{k})-f(\bm{x}^{*})\right| = \mathcal{O}(c^{-k})$.
\end{itemize}
\end{corollary}

In principle, by selecting a sufficiently fast-growing sequence $\{\sigma_k\}$, one could obtain arbitrarily fast ergodic convergence rates. However, such choices come with potential drawbacks. Rapid growth in $\{\sigma_k\}$ typically leads to increasingly ill-conditioned subproblems, making them much harder to solve numerically.

\section{Conclusion}\label{sec-conclusion}

In this work, we established convergence properties for a relative-type inexact preconditioned proximal augmented Lagrangian method (rip$^2$ALM) applied to general convex nonlinear programming. Specifically, we proved the global convergence of the sequence generated by rip$^2$ALM under standard assumptions and derived its asymptotic (super)linear convergence rate under a suitable error bound condition, providing a refined characterization of its asymptotic behavior. Moreover, we developed a novel global ergodic convergence analysis with respect to both primal feasibility violations and primal objective gaps. Collectively, these results underscore the robustness of rip$^2$ALM in practical applications and offer guidance for balancing the solution accuracy, the computational efficiency, and inexactness in subproblem minimization.

Looking ahead, the theoretical framework developed here gives rise to several interesting topics to explore. One avenue is to extend the analysis to structured problem classes, where additional structure can be leveraged to relax technical conditions and/or obtain improved convergence rates. Another topic is to investigate adaptive strategies, including data-driven and reinforcement learning approaches, for tuning hyper-parameters to further enhance practical performance. We believe that the convergence properties established in this article would provide a solid foundation for advancing the theory of inexact augmented Lagrangian methods and guiding their application to large-scale problems in machine learning, operations research, signal processing, and beyond.

%%%%%%%%%%%%%%%%%%%%%%%%%%%%%%%%%%%%%%%%%%
\section*{Acknowledgments}

\noindent The research of Lei Yang is supported in part by the National Key Research and Development Program of China under grant 2023YFB3001704, and the National Natural Science Foundation of China under grant 12301411.
The research of Kim-Chuan Toh is supported in part by the
Ministry of Education, Singapore, under its Academic Research Fund Tier 2 grant (MOE-T2EP20224-0017).

%%%%%%%%%%%%%%%%%%%%%%%%%%%%%%%%%%%%%%%%%%
\appendix

%%%%%%%%%%%%%%%%%%%%%%%%%%%%%%%%%%%%%%%%%%
\section{A technical lemma}\label{sec:proof-opt-conditions}

We present a useful lemma that will be used in our proofs.
\begin{lemma}\label{lem:opt-conditions}
The sequences
\begin{equation*}
\{\bm{x}^k\}_{k=0}^{\infty}, \quad
\{\Delta^k\}_{k=1}^{\infty}, \quad
\{\bm{w}^k\}_{k=0}^{\infty} \subset \mathbb{R}^{N},
\quad \text{and} \quad
\{\bm{y}^k:=(\bm{\lambda}^{k}, \bm{\mu}^{k})\}_{k=0}^{\infty} \subset \mathbb{R}^{M}
\end{equation*}
generated by the rip$^2$ALM in Algorithm \ref{algo:rip2ALM} satisfy, for all $k \geq 0$, the following conditions:
\begin{numcases}{}
\big(\Delta^{k+1}-\tau_k\sigma_{k}^{-1}S(\bm{x}^{k+1}-\bm{x}^{k}), \,\sigma_{k}^{-1}\big(\bm{y}^{k}-\bm{y}^{k+1}\big)\big)
\in \partial \ell\big(\bm{x}^{k+1}, \bm{y}^{k+1}\big), \label{gen-optcond1}\\[5pt]
2 \big|\langle \bm{w}^{k}-\bm{x}^{k+1}, \sigma_k\Delta^{k+1}\rangle\big|
+\big\|\sigma_{k}\Delta^{k+1}\big\|^2
\leq \rho\Big(\big\|\bm{y}^{k+1}-\bm{y}^{k}\big\|^2
+ \tau_k\big\|\bm{x}^{k+1}-\bm{x}^{k}\big\|^{2}_{S}\Big),
\label{gen-optcond2}\\[5pt]
\bm{w}^{k+1} = \bm{w}^{k}-\sigma_k \Delta^{k+1}. \label{gen-optcond3}
\end{numcases}
\end{lemma}
\begin{proof}
We first prove \eqref{gen-optcond1}. From the inexact optimality condition \eqref{rip2ALM-inexcond} and the definition of the augmented Lagrangian $\mathcal{L}_{\sigma}$ in \eqref{gen-auglag}, it follows that
\begin{equation*}
\begin{aligned}
&\quad \Delta^{k+1}-\tau_k\sigma_{k}^{-1}S(\bm{x}^{k+1}-\bm{x}^{k}) \\
&\in \partial f(\bm{x}^{k+1}) + A^{\top}\bm{\lambda}^{k} + \sigma_k A^{\top}(A\bm{x}^{k+1}-\bm{b})
+ \nabla g(\bm{x}^{k+1})^{\top}\max\{\bm{0},\,\bm{\mu}^{k}+\sigma_{k}g(\bm{x}^{k+1})\}.
\end{aligned}
\end{equation*}
By the updating rules of $\bm{\lambda}^{k+1}$ and $\bm{\mu}^{k+1}$ in \eqref{rip2ALM_ywupdate}, this relation can be further simplified to
\begin{equation*}
\Delta^{k+1}-\tau_k\sigma_{k}^{-1}S(\bm{x}^{k+1}-\bm{x}^{k})
\in \partial f(\bm{x}^{k+1}) + A^{\top}\bm{\lambda}^{k+1} + \nabla g(\bm{x}^{k+1})^{\top}\bm{\mu}^{k+1}.
\end{equation*}
Moreover, using the updating rules of $\bm{\lambda}^{k+1}$ and $\bm{\mu}^{k+1}$ again, we obtain that
\begin{equation*}
\sigma_k^{-1}(\bm{\lambda}^k-\bm{\lambda}^{k+1}) = \bm{b}-A\bm{x}^{k+1},
\qquad
\sigma_k^{-1}(\bm{\mu}^k-\bm{\mu}^{k+1})
\in -g(\bm{x}^{k+1}) + \mathcal{N}_{\mathbb{R}_+^{m_2}}(\bm{\mu}^{k+1}).
\end{equation*}
Combining the above inclusions with the characterization \eqref{Lagform} of $\partial \ell$, we conclude that
\begin{equation*}
\big(\Delta^{k+1}-\tau_k\sigma_{k}^{-1}S(\bm{x}^{k+1}-\bm{x}^{k}), \;\sigma_{k}^{-1}(\bm{y}^{k}-\bm{y}^{k+1})\big) \in \partial \ell(\bm{x}^{k+1}, \bm{y}^{k+1}),
\end{equation*}
which is \eqref{gen-optcond1}. Finally, conditions \eqref{gen-optcond2} and \eqref{gen-optcond3} directly follow from the error criterion \eqref{rip2ALM-stopcrit} and the updating rules of $\bm{\lambda}^{k+1}$, $\bm{\mu}^{k+1}$, $\bm{w}^{k+1}$ in \eqref{rip2ALM_ywupdate}. This completes the proof.
\end{proof}

%%%%%%%%%%%%%%%%%%%%%%%%%%%%%%%%%%%%%%%%%%%
\section{Proof of Theorem \ref{thm:convergence-gen}}\label{sec:proof-conver-gen}

\begin{proof}
	\textit{Statement (i)}.
	Let $(\bm{x}^*,\bm{y}^*)\in\mathbb{R}^{N}\times\mathbb{R}^{M}$ be an arbitrary saddle point of $\ell$ and hence $(\bm{0},\bm{0})\in\partial\ell(\bm{x}^*,\bm{y}^*)$. For all $k\geq0$,
	\begin{equation*}
		\begin{aligned}
			\|\bm{y}^k-\bm{y}^*\|^2
			&= \|\bm{y}^k-\bm{y}^{k+1}+\bm{y}^{k+1}-\bm{y}^*\|^2 \\
			&= \|\bm{y}^k-\bm{y}^{k+1}\|^2 + 2\langle\bm{y}^k-\bm{y}^{k+1},\,\bm{y}^{k+1}-\bm{y}^*\rangle + \|\bm{y}^{k+1}-\bm{y}^*\|^2.
		\end{aligned}
	\end{equation*}
	By letting $\bm{u}^{k+1}:=\sigma_{k}^{-1}(\bm{y}^{k}-\bm{y}^{k+1})$, the above equation can be reformulated as
	\begin{equation}\label{gen-yyineq}
		\|\bm{y}^{k+1}-\bm{y}^*\|^2 = \|\bm{y}^k-\bm{y}^*\|^2 - 2\sigma_{k}\langle\bm{y}^{k+1}-\bm{y}^*,\,\bm{u}^{k+1}\rangle - \|\bm{y}^{k+1}-\bm{y}^{k}\|^2.
	\end{equation}
	Then, using the relation $\bm{w}^{k+1} = \bm{w}^k - \sigma_{k}\Delta^{k+1}$ (by \eqref{gen-optcond3}), we see that
	\begin{equation}\label{gen-wwineq}
		\begin{aligned}
			\|\bm{w}^{k+1}-\bm{x}^*\|^2
			= &\; \|\bm{w}^k-\sigma_{k}\Delta^{k+1}-\bm{x}^*\|^2  \\
			= &\; \|\bm{w}^k-\bm{x}^*\|^2 - 2\langle\bm{w}^k-\bm{x}^*,\,\sigma_{k}\Delta^{k+1}\rangle + \|\sigma_{k}\Delta^{k+1}\|^2 \\
			= &\; \|\bm{w}^k-\bm{x}^*\|^2 - 2\langle\bm{w}^k-\bm{x}^{k+1},\,\sigma_{k}\Delta^{k+1}\rangle + \|\sigma_{k}\Delta^{k+1}\|^2 \\
			&\; - 2\sigma_{k}\langle\bm{x}^{k+1}-\bm{x}^*,\,\bm{p}^{k+1}\rangle
			- 2\tau_k\langle\bm{x}^{k+1}-\bm{x}^*,\,S(\bm{x}^{k+1}-\bm{x}^{k})\rangle,
		\end{aligned}
	\end{equation}
	where $\bm{p}^{k+1}:=\Delta^{k+1}-\tau_k\sigma_{k}^{-1}S(\bm{x}^{k+1}-\bm{x}^{k})$. Similarly,
	\begin{equation}\label{gen-xxineq}
		\hspace{-1mm}
			\tau_k\|\bm{x}^{k+1}-\bm{x}^*\|^2_{S}
			= \tau_k\|\bm{x}^{k}-\bm{x}^*\|^2_{S}
			- \tau_k\|\bm{x}^{k+1}-\bm{x}^k\|^2_{S}
			+ 2\tau_k\langle\bm{x}^{k+1}-\bm{x}^*,\,S(\bm{x}^{k+1}-\bm{x}^{k})\rangle.
	\end{equation}
	By summing \eqref{gen-yyineq}, \eqref{gen-wwineq} and \eqref{gen-xxineq}, we have that
	\begin{equation}\label{gen-ywxineq}
		\begin{aligned}
			&\; \|\bm{y}^{k+1}-\bm{y}^*\|^2 + \|\bm{w}^{k+1}-\bm{x}^*\|^2 + \tau_k\|\bm{x}^{k+1}-\bm{x}^*\|^2_{S}  \\
			= &\; \|\bm{y}^{k}-\bm{y}^*\|^2 + \|\bm{w}^{k}-\bm{x}^*\|^2 + \tau_k\|\bm{x}^k-\bm{x}^*\|^2_{S}  \\
			&\; - 2\sigma_{k}\big(\langle\bm{x}^{k+1}-\bm{x}^*,\,\bm{p}^{k+1}\rangle
			+ \langle\bm{y}^{k+1}-\bm{y}^*,\,\bm{u}^{k+1}\rangle\big) \\
			&\; - 2\langle\bm{w}^k-\bm{x}^{k+1},\,\sigma_{k}\Delta^{k+1}\rangle
			+ \|\sigma_{k}\Delta^{k+1}\|^2
			- \|\bm{y}^{k+1}-\bm{y}^{k}\|^2
			- \tau_k\|\bm{x}^{k+1}-\bm{x}^k\|^2_{S}.
		\end{aligned}
	\end{equation}
	Note from \eqref{gen-optcond1} that
	\begin{equation}\label{gen-optcond1re}
		\big(\bm{p}^{k+1}, \,\bm{u}^{k+1}\big) \in \partial \ell(\bm{x}^{k+1}, \bm{y}^{k+1}),
	\end{equation}
	which, together with $(\bm{0},\bm{0})\in\partial\ell(\bm{x}^*,\bm{y}^*)$ and the monotonicity of $\partial\ell$, yields
	\begin{equation*}
		\langle\bm{x}^{k+1}-\bm{x}^*,\,\bm{p}^{k+1}\rangle
		+ \langle\bm{y}^{k+1}-\bm{y}^*,\,\bm{u}^{k+1}\rangle \geq 0.
	\end{equation*}
	Moreover, by using \eqref{gen-optcond2}, we see that
	\begin{equation*}\label{gen-errbdineq}
		\begin{aligned}
			- 2\langle\bm{w}^k-\bm{x}^{k+1},\,\sigma_{k}\Delta^{k+1}\rangle + \|\sigma_{k}\Delta^{k+1}\|^2
			&\leq 2\big|\langle\bm{w}^k-\bm{x}^{k+1},\,\sigma_{k}\Delta^{k+1}\rangle\big| + \|\sigma_{k}\Delta^{k+1}\|^2  \\
			&\leq \rho\big(\|\bm{y}^{k+1}-\bm{y}^{k}\|^2
			+ \tau_k\|\bm{x}^{k+1}-\bm{x}^k\|^2_{S}\big).
		\end{aligned}
	\end{equation*}
	Substituting the above two inequalities into \eqref{gen-ywxineq}, we obtain a key inequality for the subsequent convergence analysis:
	\begin{equation}\label{gen-recursion_of_xyw}
		\begin{aligned}
			&\; \|\bm{y}^{k+1}-\bm{y}^*\|^2 + \|\bm{w}^{k+1}-\bm{x}^*\|^2
			+ \tau_k\|\bm{x}^{k+1}-\bm{x}^*\|^2_{S} \\
			\leq &\; \|\bm{y}^{k}-\bm{y}^*\|^2 + \|\bm{w}^{k}-\bm{x}^*\|^2
			+ \tau_k\|\bm{x}^k-\bm{x}^*\|^2_{S} \\
			&\quad - (1-\rho)\left(\|\bm{y}^{k+1}-\bm{y}^{k}\|^2
			+ \tau_k\|\bm{x}^{k+1}-\bm{x}^k\|^2_{S}\right).
		\end{aligned}
	\end{equation}
	This, together with $\rho\in[0,1)$ and $\tau_{k+1}\leq(1+\nu_k)\tau_{k}$ for all $k \geq 0$, implies that
	\begin{equation}\label{gen-relaxed_monotone}
		\begin{aligned}
			&\; \|\bm{y}^{k+1}-\bm{y}^*\|^2 + \|\bm{w}^{k+1}-\bm{x}^*\|^2 + \tau_{k+1}\|\bm{x}^{k+1}-\bm{x}^*\|^2_{S}\\
			\leq &\; (1+\nu_{k})\left(\|\bm{y}^{k}-\bm{y}^*\|^2 + \|\bm{w}^{k}-\bm{x}^*\|^2 + \tau_k\|\bm{x}^k-\bm{x}^*\|^2_{S}\right).
		\end{aligned}
	\end{equation}
	Since $\{\nu_{k}\}$ is a non-negative summable sequence, it then follows from \cite[Lemma 2 in Chapter 2.2.1]{polyak1987introduction} that the sequence $\left\{\|\bm{y}^{k}-\bm{y}^{*}\|^2+\|\bm{w}^{k}-\bm{x}^{*}\|^{2} + \tau_k\|\bm{x}^k-\bm{x}^{*}\|^2_{S}\right\}$ is convergent. This, together with $\tau_{k}\geq\tau_{\min}>0$ and the positive definiteness of $S$, implies that all sequences $\{\bm{x}^k\}$, $\{\bm{w}^k\}$ and $\{\bm{y}^{k}\}$ are bounded.
	
	\vspace{2mm}
	\textit{Statement (ii)}. Using \eqref{gen-recursion_of_xyw} again with $\tau_{k+1}\leq(1+\nu_k)\tau_{k}$ with $\nu_k\geq0$ and $\sum\nu_i<\infty$ for all $k \geq 0$, we have that
	\begin{equation}\label{gen-addineq1}
		\begin{aligned}
			0 \leq &\; (1-\rho){(1+\nu_k)}\left(\|\bm{y}^{k+1}-\bm{y}^{k}\|^2 + \tau_k\|\bm{x}^{k+1}-\bm{x}^k\|^2_{S}\right) \\
			\leq &\; (1+\nu_{k})\left(\|\bm{y}^{k}-\bm{y}^*\|^2 + \|\bm{w}^{k}-\bm{x}^*\|^2 + \tau_k\|\bm{x}^k-\bm{x}^*\|^2_{S}\right) \\
			&\quad - \left(\|\bm{y}^{k+1}-\bm{y}^*\|^2 + \|\bm{w}^{k+1}-\bm{x}^*\|^2 + \tau_{k+1}\|\bm{x}^{k+1}-\bm{x}^*\|^2_{S}\right).
		\end{aligned}
	\end{equation}
	Since $\left\{\|\bm{y}^{k}-\bm{y}^{*}\|^2+\|\bm{w}^{k}-\bm{x}^{*}\|^{2}+\tau_k\|\bm{x}^k-\bm{x}^{*}\|^2_{S}\right\}$ is convergent, $\nu_{k}\to0$ (due to $\nu_k\geq0$ and $\sum\nu_i<\infty$) and $\rho\in[0,1)$, it then follows from \eqref{gen-addineq1} that
	\begin{equation}\label{gen-conv_succhg_xy}
		\lim\limits_{k\to\infty}~\|\bm{y}^{k+1}-\bm{y}^{k}\| = 0
		\quad \mathrm{and} \quad \lim\limits_{k\to\infty}~\tau_k\|\bm{x}^{k+1}-\bm{x}^k\| _{S}= 0.
	\end{equation}
	Note also that $\{\sigma_k\}$ is bounded away from 0, $\tau_k\geq\tau_{\min}>0$ and $S$ is positive definite. Thus, we further have that $\lim\limits_{k\to\infty}\bm{u}^{k+1} (:= \sigma_{k}^{-1}(\bm{y}^{k}-\bm{y}^{k+1}))=\bm{0}$ and $\lim\limits_{k\to\infty}\|\bm{x}^{k+1}-\bm{x}^k\|=0$. Moreover, using \eqref{gen-conv_succhg_xy} together with \eqref{gen-optcond2} implies that
	\begin{equation*}%\label{eq:errlim}
		\lim\limits_{k\to\infty}~|\langle\bm{w}^k-\bm{x}^{k+1},\,\sigma_k\Delta^{k+1}\rangle|
		= 0
		\quad \mathrm{and} \quad
		\lim\limits_{k\to\infty}~\|\sigma_k\Delta^{k+1}\|^2 = 0.
	\end{equation*}
	Since $\{\sigma_{k}\}$ is bounded away from $0$, we then obtain that $\lim\limits_{k\to\infty}\langle\bm{w}^k-\bm{x}^{k+1},\,\Delta^{k+1}\rangle=0$ and $\lim\limits_{k\to\infty}\Delta^{k+1}=\bm{0}$. Finally, recall again that $\tau_{k+1}\leq (1+\nu_k)\tau_k$ with $\nu_k\geq0$ and $\sum\nu_i<\infty$ for all $k\geq0$. Thus, $\tau_{k}$ must be bounded from above and hence $\tau_{k}\sigma_{k}^{-1}$ is also bounded from above. Consequently, we can obtain that $\lim\limits_{k\to\infty}\tau_{k}\sigma_{k}^{-1}S(\bm{x}^{k+1}-\bm{x}^{k})=\bm{0}$ and hence $\lim\limits_{k\to\infty}\,\bm{p}^{k+1}\,(:=\Delta^{k+1}-\tau_k\sigma_{k}^{-1}S(\bm{x}^{k+1}-\bm{x}^{k}))=\bm{0}$.
	
	\vspace{1mm}
	\textit{Statement (iii)}. We first study the limit of $\{G\big(\bm{p}^{k+1}, \,\bm{y}^{k+1}\big)\}$. Using relations \eqref{gen-optcond1re} and \eqref{gen-subdiffrelation}, we have that $(\bm{x}^{k+1}, \bm{u}^{k+1})\in\partial G\big(\bm{p}^{k+1}, \,\bm{y}^{k+1}\big)$. Then, by the concavity of $G$, it holds that, for all $k\geq 0$,
	\begin{equation*}
		G(\bm{0},\,\bm{y}^*) \leq G\big(\bm{p}^{k+1}, \,\bm{y}^{k+1}\big)
		+ \langle\bm{x}^{k+1}, \,\bm{p}^{k+1}\rangle
		+ \langle\bm{u}^{k+1}, \,\bm{y}^{k+1}-\bm{y}^*\rangle.
	\end{equation*}
	Since $\lim\limits_{k\to\infty}\bm{p}^{k+1}=\bm{0}$, $\lim\limits_{k\to\infty}\bm{u}^{k+1} = \bm{0}$, and the sequences $\{\bm{x}^k\}$ and $\{\bm{y}^{k}\}$ are bounded, we can obtain from the above inequality that
	\begin{equation}\label{gen-liminfres}
		\liminf\limits_{k\to\infty}\,G\big(\bm{p}^{k+1}, \,\bm{y}^{k+1}\big)
		\geq G(\bm{0},\,\bm{y}^*).
	\end{equation}
	On the other hand, since $\{\bm{y}^k\}$ is bounded, it has at least one accumulation point. Suppose that $\bm{y}^{\infty}$ is an accumulation point and $\{\bm{y}^{k_i}\}$ is a convergent subsequence such that $\lim\limits_{i\to\infty}\bm{y}^{k_i} = \bm{y}^{\infty}$. Since $\lim\limits_{k\to\infty}\|\bm{y}^{k+1} - \bm{y}^{k}\| = 0$, we also have that $\lim\limits_{i\to\infty}\bm{y}^{k_{i}+1} = \bm{y}^{\infty}$. Thus, by passing to a further subsequence if necessary, we may assume without loss of generality that the subsequence  $\big\{G\big(\bm{p}^{k_{i}+1}, \,\bm{y}^{k_{i}+1}\big)\big\}$ satisfies
	\begin{equation*}
		\lim\limits_{i\to\infty}G\big(\bm{p}^{k_{i}+1},\bm{y}^{k_{i}+1}\big)
		=\limsup\limits_{k\to\infty}\,G\big(\bm{p}^{k+1}, \,\bm{y}^{k+1}\big).
	\end{equation*}
	Note that $G$ is closed upper semicontinuous concave (see, for example, \cite[Theorem 7]{r1974conjugate}), and thus $\mathrm{dom}\,G$ is closed. This, together with $(\bm{0},\,\bm{y}^{\infty})=\lim\limits_{i\to\infty}\big(\bm{p}^{k_{i}+1}, \,\bm{y}^{k_{i}+1}\big)$, induces that $(\bm{0},\,\bm{y}^{\infty})\in\mathrm{dom}\,G$. Then, we see that
	\begin{equation*}
		\begin{aligned}
			G(\bm{0},\,\bm{y}^*)
			&\geq G(\bm{0},\,\bm{y}^{\infty}) && \quad \mbox{(since $\bm{y}^*$ is optimal for dual problem \eqref{gen-paraobjdual})} \\
			&= G\bigg(\lim\limits_{i\to\infty}\bm{p}^{k_{i}+1},\,\lim\limits_{i\to\infty}\bm{y}^{k_{i}+1}\bigg) && \quad \mbox{(since $\lim\limits_{k\to\infty}\bm{p}^{k+1}=\bm{0}$ and $\lim\limits_{i\to\infty}\bm{y}^{k_{i}+1} = \bm{y}^{\infty}$)} \\
			&\geq \limsup\limits_{i\to\infty}\,G\big(\bm{p}^{k_{i}+1}, \,\bm{y}^{k_{i}+1}\big)
			&& \quad \mbox{(since $G$ is upper semicontinuous)} \\
			&= \limsup\limits_{k\to\infty}\,G\big(\bm{p}^{k+1}, \,\bm{y}^{k+1}\big).  && \quad \mbox{(by the choice of subsequence $\{\bm{y}^{k_{i}+1}\}$)}
		\end{aligned}
	\end{equation*}
	This, together with \eqref{gen-liminfres}, implies that
	\begin{equation*}%\label{gen-limofG}
		\lim\limits_{k\to\infty}\,G\big(\bm{p}^{k+1}, \,\bm{y}^{k+1}\big)
		= G(\bm{0},\,\bm{y}^*).
	\end{equation*}
	
	We next study the limit of $\{F(\bm{x}^{k+1}, \bm{u}^{k+1})\}$. Since $-F$ and $G$ are convex conjugate and $\big(\bm{p}^{k+1},\,\bm{y}^{k+1}\big)\in\partial F(\bm{x}^{k+1},\,\bm{u}^{k+1})$, we can get the following equality by using the Fenchel equality (see, for example, \cite[Theorem 23.5]{r1970convex}):
	\begin{equation*}
		F(\bm{x}^{k+1}, \,\bm{u}^{k+1})
		= G\big(\bm{p}^{k+1}, \,\bm{y}^{k+1}\big)
		+ \langle\bm{p}^{k+1}, \,\bm{x}^{k+1}\rangle
		+ \langle\bm{y}^{k+1}, \,\bm{u}^{k+1}\rangle.
	\end{equation*}
	Since $\lim\limits_{k\to\infty}\bm{p}^{k+1}=\bm{0}$, $\lim\limits_{k\to\infty}\bm{u}^{k+1} = \bm{0}$, and $\{\bm{x}^k\}$ and $\{\bm{y}^{k}\}$ are bounded, we obtain that
	\begin{equation*}
		\lim\limits_{k\to\infty}\,F(\bm{x}^{k+1}, \bm{u}^{k+1})
		= G(\bm{0},\,\bm{y}^*) = F(\bm{x}^*,\,\bm{0}).
	\end{equation*}
	This proves statement (iii).
	
	\vspace{2mm}
	\textit{Statement (iv)}. We first prove that any accumulation point of $\{\bm{x}^k\}$ is an optimal solution of problem \eqref{gen-paraobjpri}. Since $\{\bm{x}^k\}$ is bounded by statement (i), the sequence $\{\bm{x}^k\}$ has at least one accumulation point. Suppose that $\bm{x}^{\infty}$ is an accumulation point and $\{\bm{x}^{k_j}\}$ is a convergent subsequence such that $\lim\limits_{j\to\infty}\bm{x}^{k_j} = \bm{x}^{\infty}$. Since $\lim\limits_{k\to\infty}\|\bm{x}^{k+1}-\bm{x}^{k}\| = 0$, we also have that $\lim\limits_{j\to\infty}\bm{x}^{k_j+1} = \bm{x}^{\infty}$. Then, using the fact that $F$ is lower semicontinuous and convex, and $\lim\limits_{k\to\infty}\bm{u}^{k+1}=\bm{0}$, we obtain that
	\begin{equation*}
		F(\bm{x}^{\infty},\bm{0})
		= F(\lim_{j\to\infty}\bm{x}^{k_j+1},\,\lim_{j\to\infty}\bm{u}^{k_{j}+1})
		\leq \liminf_{j\to\infty}~F(\bm{x}^{k_j+1},\,\bm{u}^{k_j+1})
		= F(\bm{x}^*,\bm{0}).
	\end{equation*}
	This implies that $\bm{x}^{\infty}$ is an optimal solution of problem \eqref{gen-paraobjpri}. Similarly, using the upper semicontinuity of $G$ and analogous manipulations, we can prove that any accumulation point of $\{\bm{y}^k\}$ is an optimal solution of problem \eqref{gen-paraobjdual}. This proves statement (iv).
	
	\vspace{2mm}
	\textit{Statement (v)}. We next prove that the whole sequence $\{\bm{y}^{k}\}$ is convergent. Define
	\begin{equation*}
		D_{\tau_{k}}\left((\bm{w}^k,\bm{x}^k), \,\mathcal{X}^*\right)
		:= \inf\limits_{\bm{x}^*\in\mathcal{X}^*}\left\{\|\bm{w}^{k}-\bm{x}^*\|^2
		+ \tau_{k}\|\bm{x}^{k}-\bm{x}^*\|^2_{S}\right\},
	\end{equation*}
	and let
	\begin{equation*}
		\phi := \liminf_{k\to\infty}\,D_{\tau_{k}}\left((\bm{w}^k,\bm{x}^k), \,\mathcal{X}^*\right),
	\end{equation*}
	where $\mathcal{X}^*$ is the solution set of primal problem \eqref{gen-paraobjpri}. Since $\{\bm{w}^k\}$, $\{\bm{x}^{k}\}$ and $\{\tau_{k}\}$ are bounded, we see that $0 \leq \phi < \infty$ and there exists a subsequence $\{(\bm{w}^{k_j},\bm{x}^{k_j},\tau_{k_j})\}$ such that
	\begin{equation*}
\lim_{j\to\infty}\,D_{\tau_{k_j}}\left((\bm{w}^{k_j},\bm{x}^{k_j}), \,\mathcal{X}^*\right) = \phi.
	\end{equation*}
	Then, by passing to a further subsequence if necessary, we may also assume without loss of generality that the subsequence $\{\bm{y}^{k_j}\}\subseteq\{\bm{y}^k\}$ converges to some accumulation point $\bm{y}^{\infty}$, which, in view of statement (iv), belongs to $\mathcal{Y}^*$ (the solution set of dual problem \eqref{gen-paraobjdual}). Thus, for such $\bm{y}^{\infty}$ and any $\bm{x}^*\in\mathcal{X}^*$, using \eqref{gen-relaxed_monotone} with some manipulations, we can obtain that, for all $k > k_j$,
	\begin{equation*}
		\begin{aligned}
			&\,\|\bm{y}^{k}-\bm{y}^{\infty}\|^2 + \|\bm{w}^{k}-\bm{x}^*\|^2
			+ \tau_{k}\|\bm{x}^{k}-\bm{x}^*\|^2_{S} \\
			\leq & \; {\textstyle\left(\prod_{i=k_j}^{k-1}\big(1+\nu_i\big)\right)}
			\Big(\|\bm{y}^{k_j}-\bm{y}^{\infty}\|^2 + \|\bm{w}^{k_j}-\bm{x}^*\|^2
			+ \tau_{k_j}\|\bm{x}^{k_{j}}-\bm{x}^*\|^2_{S}\Big).
		\end{aligned}
	\end{equation*}
	Since $0 \leq \phi \leq \liminf\limits_{k\to\infty}\left\{\|\bm{w}^{k}-\bm{x}^*\|^2 + \tau_{k}\|\bm{x}^{k}-\bm{x}^*\|^2_{S}\right\}$ for any $\bm{x}^*\in\mathcal{X}^*$, passing to the limit superior when $k \to \infty$ on the both sides of the above inequality, we obtain that, for any $\bm{x}^*\in\mathcal{X}^*$,
	\begin{equation*}
		\begin{aligned}
			&\; \phi + \limsup\limits_{k\to\infty}\left\{\|\bm{y}^{k}-\bm{y}^{\infty}\|^2\right\} \\
			\leq &\; \limsup\limits_{k\to\infty}\left\{\|\bm{y}^{k}-\bm{y}^{\infty}\|^2 + \|\bm{w}^{k}-\bm{x}^*\|^2 + \tau_{k}\|\bm{x}^{k}-\bm{x}^*\|^2_{S}\right\} \\
			\leq &\; {\textstyle\left(\prod_{i=k_j}^{\infty}(1+\nu_i)\right)}
			\left(\|\bm{y}^{k_j}-\bm{y}^{\infty}\|^2 + \|\bm{w}^{k_j}-\bm{x}^*\|^2
			+ \tau_{k_j}\|\bm{x}^{k_j}-\bm{x}^*\|^2_{S}\right), \quad \forall j \geq 0.
		\end{aligned}
	\end{equation*}
	Taking the infimum in $\bm{x}^*\in\mathcal{X}^*$ on the right-hand side of the last inequality, we have that
	\begin{equation*}
		\begin{aligned}		
			&\limsup\limits_{k\to\infty}\left\{\|\bm{y}^{k}-\bm{y}^{\infty}\|^2\right\} \\
			\leq &\; {\textstyle\left(\prod_{i=k_j}^{\infty}(1+\nu_i)\right)}
			\|\bm{y}^{k_j}-\bm{y}^{\infty}\|^2
			+ {\textstyle\left(\prod_{i=k_j}^{\infty}(1+\nu_i)\right)}
			D_{\tau_{k_j}}\left((\bm{w}^{k_j},\bm{x}^{k_j}), \,\mathcal{X}\right)
			- \phi, \quad \forall j \geq 0.
		\end{aligned}
	\end{equation*}
	Since $\ln\left(\prod_{i=k_{j}}^{\infty}(1+\nu_{i})\right) = \sum_{i=k_{j}}^{\infty}\ln(1+\nu_{i}) \leq \sum_{i=k_{j}}^{\infty}\nu_{i}$ and $\lim\limits_{j\to\infty}\sum_{i=k_{j}}^{\infty}\nu_{i} = 0$ (due to the summability of $\{\nu_k\}$), we see that $\lim\limits_{j\to\infty}\prod_{i=k_{j}}^{\infty}(1+\nu_{i}) = 1$. Using this fact, we can observe that the right-hand side of the above inequality converges to 0 as $j\to\infty$. Then, we conclude that $\lim\limits_{k\to\infty} \bm{y}^{k} = \bm{y}^{\infty}$, which completes the proof.
\end{proof}

%%%%%%%%%%%%%%%%%%%%%%%%%%%%%%%%%%%%%%%%%%
\section{Proof of Theorem \ref{thm:Q-linear-rate-gen}}\label{sec:proof-Q-linear-rate-gen}

\begin{proof}
	\textit{Statement (i)}. For the sake of clarity, we will present our proof in three steps.
	
	\vspace{2mm}
	\textbf{Step I.} Let $(\bm{x}^*,\bm{y}^*)\in\mathbb{R}^N\times\mathbb{R}^M$ be an arbitrary saddle point of $\ell$. Similar to the proof of statement (i) in Theorem \ref{thm:convergence-gen}, we combine \eqref{gen-yyineq} with \eqref{gen-xxineq} to obtain that
	\begin{equation*}
		\begin{aligned}
			&\; \|\bm{y}^{k+1}-\bm{y}^*\|^2 + \tau_{k}\|\bm{x}^{k+1}-\bm{x}^*\|^2_{S}  \\
			= &\; \|\bm{y}^{k}-\bm{y}^*\|^2 + \tau_{k}\|\bm{x}^k-\bm{x}^*\|^2_{S}
			\; - 2\sigma_{k}\Big(\langle\bm{x}^{k+1}-\bm{x}^*,\,\bm{p}^{k+1}\rangle
			+ \langle\bm{y}^{k+1}-\bm{y}^{*},\,\bm{u}^{k+1}\rangle\Big) \\
			&\; + 2\sigma_{k}\langle\bm{x}^{k+1}-\bm{x}^{*}, \Delta^{k+1}\rangle
			- \|\bm{y}^{k+1}-\bm{y}^{k}\|^2 - \tau_{k}\|\bm{x}^{k+1}-\bm{x}^k\|^2_{S}.
		\end{aligned}
	\end{equation*}
	Since $(\bm{0},\bm{0})\in\partial\ell(\bm{x}^*,\bm{y}^*)$ and $\left(\bm{p}^{k+1}, \,\bm{u}^{k+1}\right) \in \partial \ell(\bm{x}^{k+1}, \bm{y}^{k+1})$, it then follows from the monotonicity of $\partial\ell$ that
	\begin{equation*}
		\langle\bm{x}^{k+1}-\bm{x}^*,\,\bm{p}^{k+1}\rangle
		+ \langle\bm{y}^{k+1}-\bm{y}^{*},\,\bm{u}^{k+1}\rangle \geq 0.
	\end{equation*}
	Thus, we conclude that
	\begin{equation}\label{gen-succ-xychg}
		\left\|\begin{matrix}
			\bm{x}^{k}-\bm{x}^{*} \\
			\bm{y}^{k} - \bm{y}^{*}
		\end{matrix}\right\|^{2}_{\Lambda_k}
		- \left\|\begin{matrix}
			\bm{x}^{k+1}-\bm{x}^{*} \\
			\bm{y}^{k+1} - \bm{y}^{*}
		\end{matrix} \right\|^{2}_{\Lambda_k}
		\geq \left\|\begin{matrix}
			\bm{x}^{k+1}-\bm{x}^{k} \\
			\bm{y}^{k+1} - \bm{y}^{k}
		\end{matrix} \right\|^2_{\Lambda_k}
		- 2\sigma_{k}\|\bm{x}^{k+1}-\bm{x}^{*}\|\|\Delta^{k+1}\|,
	\end{equation}
	where $\Lambda_k := \mathrm{Diag}(\tau_{k}S,I_{M})$. Define the sequences $\{\overline{\bm{x}}^{k}\}\subseteq\mathbb{R}^N$ and $\{\overline{\bm{y}}^{k}\}\subseteq\mathbb{R}^M$ as follows:
	\begin{equation*}\label{gen-def-of-proj}
		\overline{\bm{x}}^{k} := \Pi_{\mathcal{X}^*,S}(\bm{x}^{k})
		\quad \text{and} \quad
		\overline{\bm{y}}^{k} := \Pi_{\mathcal{Y}^*}(\bm{y}^{k}),\quad \forall\; k\geq 0,
	\end{equation*}
	where $\mathcal{X}^*$ is the solution set of primal problem \eqref{gen-paraobjpri} (i.e., problem \eqref{gen-mainprob}), $\mathcal{Y}^*$ is the solution set of dual problem \eqref{gen-paraobjdual}, $\Pi_{\mathcal{X}^*,S}(\bm{x}^{k})$ denotes the weighted projection of $\bm{x}^{k}$ onto the set $\mathcal{X}^*$, and $\Pi_{\mathcal{Y}^*}(\bm{y}^{k})$ denotes the Euclidean projection of $\bm{y}^k$ onto the set $\mathcal{Y}^*$. Since \eqref{gen-succ-xychg} holds for any $\bm{x}^{*}\in\mathcal{X}^*$ and $\bm{y}^{*}\in\mathcal{Y}^*$, we can replace $\bm{x}^{*}$ and $\bm{y}^{*}$ with $\overline{\bm{x}}^{k}$ and $\overline{\bm{y}}^{k}$, respectively, to obtain
	\begin{equation}\label{eq-recur-tmp}
		\left\|\begin{matrix}
			\bm{x}^{k}-\overline{\bm{x}}^{k} \\
			\bm{y}^{k} - \overline{\bm{y}}^{k}
		\end{matrix} \right\|^{2}_{\Lambda_k}
		-
		\left\|\begin{matrix}
			\bm{x}^{k+1}-\overline{\bm{x}}^{k} \\
			\bm{y}^{k+1} - \overline{\bm{y}}^{k}
		\end{matrix} \right\|^{2}_{\Lambda_k}
		\geq
		\left\|\begin{matrix}
			\bm{x}^{k+1}-\bm{x}^{k} \\
			\bm{y}^{k+1} - \bm{y}^{k}
		\end{matrix} \right\|^2_{\Lambda_k}
		- 2\sigma_{k}\|\bm{x}^{k+1}-\overline{\bm{x}}^{k}\|\|\Delta^{k+1}\|.
	\end{equation}
	Moreover, from the definitions of $\overline{\bm{x}}^{k}$ and $\overline{\bm{y}}^{k}$, together with $\tau_{k+1}\leq(1+\nu_{k})\tau_{k}$, we have that
	\begin{equation*}
		\tau_{k+1}\big\|\bm{x}^{k+1} - \overline{\bm{x}}^{k+1}\big\|^2_{S}
		\leq \tau_{k+1}\big\|\bm{x}^{k+1} - \overline{\bm{x}}^{k}\big\|^2_{S}
		\leq (1+\nu_k)\tau_{k}\big\|\bm{x}^{k+1} - \overline{\bm{x}}^{k}\big\|^2_{S},
	\end{equation*}
	and
	\begin{equation*}
		\big\|\bm{y}^{k+1} - \overline{\bm{y}}^{k+1}\big\|^2
		\leq \big\|\bm{y}^{k+1} - \overline{\bm{y}}^{k}\big\|^2,
	\end{equation*}
	which further imply that
	\begin{equation*}
		\frac{1}{1+\nu_k}\left\|\begin{matrix}
			\bm{x}^{k+1}-\overline{\bm{x}}^{k+1} \\
			\bm{y}^{k+1} - \overline{\bm{y}}^{k+1}
		\end{matrix} \right\|^{2}_{\Lambda_{k+1}}
		\leq
		\left\|\begin{matrix}
			\bm{x}^{k+1}-\overline{\bm{x}}^{k} \\
			\bm{y}^{k+1} - \overline{\bm{y}}^{k}
		\end{matrix} \right\|^{2}_{\Lambda_k}.
	\end{equation*}
	These, together with \eqref{eq-recur-tmp}, yield that
	\begin{equation}\label{gen-succ-xychgbar}
		\begin{aligned}
			&\; \left\|\begin{matrix}
				\bm{x}^{k}-\overline{\bm{x}}^{k} \\
				\bm{y}^{k} - \overline{\bm{y}}^{k}
			\end{matrix} \right\|^{2}_{\Lambda_k}
			-
			\frac{1}{1+\nu_k}\left\|\begin{matrix}
				\bm{x}^{k+1}-\overline{\bm{x}}^{k+1} \\
				\bm{y}^{k+1} - \overline{\bm{y}}^{k+1}
			\end{matrix} \right\|^{2}_{\Lambda_{k+1}}   \\
			\geq &\; \left\|\begin{matrix}
				\bm{x}^{k+1}-\bm{x}^{k} \\
				\bm{y}^{k+1} - \bm{y}^{k}
			\end{matrix} \right\|^2_{\Lambda_k}
			- 2\sigma_{k}\|\bm{x}^{k+1}-\overline{\bm{x}}^{k}\|\|\Delta^{k+1}\|.
		\end{aligned}
	\end{equation}
	
	\vspace{2mm}
	\textbf{Step II.} We next derive an upper bound for $\|\bm{x}^{k+1}-\overline{\bm{x}}^{k}\|\|\Delta^{k+1}\|$. On the one hand, we have from \eqref{gen-optcond2} that
	\begin{equation*}
		\sigma_{k}^{2}\|\Delta^{k+1}\|^2
		\leq \rho\big(\|\bm{y}^{k+1}-\bm{y}^{k}\|^2
		+ \tau_{k}\|\bm{x}^{k+1}-\bm{x}^k\|^2_{S}\big),
	\end{equation*}
	which implies that
	\begin{equation}\label{gen-delta-upbd-xy}
		\|\Delta^{k+1}\|
		\leq \frac{\sqrt{\rho}}{\sigma_{k}}
		\left\|\begin{matrix}
			\bm{x}^{k+1}-\bm{x}^{k} \\
			\bm{y}^{k+1} - \bm{y}^{k}
		\end{matrix}\right\|_{\Lambda_k}.
	\end{equation}
	On the other hand, we see that
	\begin{equation}\label{gen-upbd-xxbark}
		\begin{aligned}
			\|\bm{x}^{k+1}-\overline{\bm{x}}^{k}\| & \leq \frac{1}{\sqrt{\lambda_{\min}(S)}}\|\bm{x}^{k+1}-\overline{\bm{x}}^{k}\|_{S} \\
			& \leq \frac{1}{\sqrt{\lambda_{\min}(S)}} \left(\|\bm{x}^{k+1}-\overline{\bm{x}}^{k+1}\|_{S}
			+ \|\overline{\bm{x}}^{k+1}-\overline{\bm{x}}^{k}\|_{S}\right) \\
			& \leq \frac{1}{\sqrt{\lambda_{\min}(S)}} \left(\|\bm{x}^{k+1}-\overline{\bm{x}}^{k+1}\|_{S}
			+ \|\bm{x}^{k+1}-{\bm{x}}^{k}\|_{S}\right) \\
			& \leq \frac{1}{\sqrt{\lambda_{\min}(S)}} \left(\|\bm{x}^{k+1}-\overline{\bm{x}}^{k+1}\|_{S}
			+ \frac{1}{\sqrt{\tau_{k}}}\left\|\begin{matrix}\bm{x}^{k+1}-{\bm{x}}^{k} \\
				\bm{y}^{k+1} - {\bm{y}}^{k}
			\end{matrix}\right\|_{\Lambda_k}\right),
		\end{aligned}
	\end{equation}
	where $\lambda_{\min}(S)$ is the smallest eigenvalue of $S$, and the third inequality follows from the non-expansiveness of the weighted projection operator $\Pi_{\mathcal{X}^*,S}(\cdot)$. Moreover, since $\{\bm{x}^{k}\}$ and $\{\bm{y}^{k}\}$ are bounded (by Theorem \ref{thm:convergence-gen}(i)), there must exists a positive scalar $r$ such that
	\begin{equation*} \mathrm{dist}\big((\bm{x}^{k},\bm{y}^{k}),\,(\partial\ell)^{-1}(\bm{0},\bm{0})\big)\leq r, \quad \forall \,k \geq 0.
	\end{equation*}
	Thus, we apply Assumption \ref{asp:error-bound-Li-gen} with this $r$ and know that, there exists a $\kappa>0$ such that
	\begin{equation}\label{eq:error-bound-xyk}
		\mathrm{dist}\big((\bm{x}^{k+1},\bm{y}^{k+1}), \,\mathcal{X}^*\times\mathcal{Y}^*\big)
		\leq \kappa\,\mathrm{dist}\big((\bm{0},\bm{0}), \,\partial\ell(\bm{x}^{k+1},\bm{y}^{k+1})\big)
		\leq \kappa\big\|(\bm{p}^{k+1}, \,\bm{u}^{k+1})\big\|,
	\end{equation}
	where the last inequality is due to \eqref{gen-optcond1re} (i.e., $(\bm{p}^{k+1}, \,\bm{u}^{k+1})\in\partial\ell(\bm{x}^{k+1}, \bm{y}^{k+1})$).  For any $k\geq0$, let $\widetilde{\bm{x}}^{k} := \Pi_{\mathcal{X}^*}(\bm{x}^{k})$ be the Euclidean projection of $\bm{x}^{k}$ onto the set $\mathcal{X}^{*}$.
	%Define sequence $\{\widetilde{\bm{x}}^{k}\}\subseteq\mathbb{R}^{N}$ as follows
	%\begin{equation*}
		%\widetilde{\bm{x}}^{k} := \Pi_{\mathcal{X}^*}(\bm{x}^{k}), \quad \forall k \geq 0,
	%\end{equation*}
	%where $\Pi_{\mathcal{X}^*}(\bm{x}^{k})$ denotes the Euclidean projection of $\bm{x}^{k}$ onto set $\mathcal{X}^{*}$.
	Then, using \eqref{eq:error-bound-xyk} with $\bm{p}^{k+1} := \Delta^{k+1}-\tau_{k}\sigma_{k}^{-1}S(\bm{x}^{k+1}-\bm{x}^{k})$ and $\bm{u}^{k+1} := \sigma_{k}^{-1}(\bm{y}^{k}-\bm{y}^{k+1})$, we see that
	\begin{equation}\label{gen-upbd-xdist}
		\begin{aligned}
			& \; \frac{1}{\sqrt{\lambda_{\max}(S)}}\|\bm{x}^{k+1} - \overline{\bm{x}}^{k+1}\|_{S}
            \leq 
            \frac{1}{\sqrt{\lambda_{\max}(S)}}\|\bm{x}^{k+1} - \widetilde{\bm{x}}^{k+1}\|_{S}
            \\ 
			\leq &\; \|\bm{x}^{k+1} - \widetilde{\bm{x}}^{k+1}\|
			\leq \; \sqrt{\|\bm{x}^{k+1} - \widetilde{\bm{x}}^{k+1}\|^2
				+ \|\bm{y}^{k+1} - \overline{\bm{y}}^{k+1}\|^2} \\
			\leq &\; \kappa\left\|\begin{matrix}
				\Delta^{k+1} - \tau_{k}\sigma_{k}^{-1}S(\bm{x}^{k+1} - \bm{x}^{k}) \\
				\sigma_{k}^{-1}(\bm{y}^{k} - \bm{y}^{k+1})
			\end{matrix}\right\|
			\;\leq \; \kappa\left(\|\Delta^{k+1}\|
			+ \frac{1}{\sigma_{k}}\left\|\begin{matrix}
				\tau_{k}S(\bm{x}^{k+1} - \bm{x}^{k}) \\
				\bm{y}^{k+1} - \bm{y}^{k}
			\end{matrix}\right\|\right)  \\
			\leq &\; \kappa\left(\|\Delta^{k+1}\|
			+ \frac{\sqrt{\overline{\tau}_{k}}}{\sigma_k}
			\left\|\begin{matrix}
				\bm{x}^{k+1} - \bm{x}^{k} \\
				\bm{y}^{k+1} - \bm{y}^{k}
			\end{matrix}\right\|_{\Lambda_k}\right)
			\leq \frac{\kappa\left(\sqrt{\rho}+\sqrt{\overline{\tau}_{k}}\right)}{\sigma_{k}}
			\left\|\begin{matrix}
				\bm{x}^{k+1}-\bm{x}^{k} \\
				\bm{y}^{k+1} - \bm{y}^{k}
			\end{matrix}\right\|_{\Lambda_k},
		\end{aligned}
	\end{equation}
	where $\overline{\tau}_{k}:=\max\left\{1, \tau_{k}\lambda_{\max}(S)\right\}$, and $\lambda_{\max}(S)$ is the largest eigenvalue of $S$. Note that the first inequality follows from the minimization property of $\overline{\bm{x}}^{k+1}$,
    %the first inequality follows from the fact that
	%\begin{equation*}
%		\frac{1}{\sqrt{\lambda_{\max}(S)}}\mathrm{dist}_{S}\big(\bm{x}^{k+1},\mathcal{X}^{*}\big) \leq \mathrm{dist}\big(\bm{x}^{k+1},\mathcal{X}^{*}\big),
%	\end{equation*}
	and the last inequality follows from \eqref{gen-delta-upbd-xy}. 
    
    Now, combining \eqref{gen-delta-upbd-xy}, \eqref{gen-upbd-xxbark} and \eqref{gen-upbd-xdist}, with some manipulations, we can obtain that
	\begin{equation*}
		\begin{aligned}
			& \; \|\bm{x}^{k+1}-\overline{\bm{x}}^{k}\|\|\Delta^{k+1}\| \\
			\overset{\eqref{gen-upbd-xxbark}}{\leq} & \;  \frac{1}{\sqrt{\lambda_{\min}(S)}}\left(\|\bm{x}^{k+1}-\overline{\bm{x}}^{k+1}\|_{S}
			+ \frac{1}{\sqrt{\tau_{k}}}\left\|\begin{matrix}\bm{x}^{k+1}-{\bm{x}}^{k} \\
				\bm{y}^{k+1} - {\bm{y}}^{k}
			\end{matrix}\right\|_{\Lambda_k}\right)\|\Delta^{k+1}\| \\
			\overset{\eqref{gen-upbd-xdist}}{\leq} & \;  \frac{1}{\sqrt{\lambda_{\min}(S)}}\left(\frac{\kappa\sqrt{\lambda_{\max}(S)}\left(\sqrt{\rho}+\sqrt{\overline{\tau}_{k}}\right)}{\sigma_{k}}
			\left\|\begin{matrix}
				\bm{x}^{k+1}-\bm{x}^{k} \\
				\bm{y}^{k+1} - \bm{y}^{k}
			\end{matrix}\right\|_{\Lambda_k}
			+ \frac{1}{\sqrt{\tau_{k}}}\left\|\begin{matrix}\bm{x}^{k+1}-{\bm{x}}^{k} \\
				\bm{y}^{k+1} - {\bm{y}}^{k}
			\end{matrix}\right\|_{\Lambda_k}\right)\|\Delta^{k+1}\| \\
			= & \; \frac{1}{\sqrt{\lambda_{\min}(S)}}\left(\frac{\kappa\sqrt{\tau_{k}\lambda_{\max}(S)}\left(\sqrt{\rho}+\sqrt{\overline{\tau}_{k}}\right) + \sigma_{k}}{\sigma_{k}\sqrt{\tau_{k}}}
			\left\|\begin{matrix}
				\bm{x}^{k+1}-\bm{x}^{k} \\
				\bm{y}^{k+1} - \bm{y}^{k}
			\end{matrix}\right\|_{\Lambda_k}\right)\|\Delta^{k+1}\| \\
			\overset{\eqref{gen-delta-upbd-xy}}{\leq} & \; \frac{1}{\sqrt{\lambda_{\min}(S)}}\left(\frac{\kappa\sqrt{\tau_{k}\lambda_{\max}(S)}
				\left(\rho +\sqrt{\rho\overline{\tau}_{k}}\right)
				+ \sigma_{k}\sqrt{\rho}}{\sigma_{k}^{2}\sqrt{\tau_{k}}}\right)
			\left\|\begin{matrix}
				\bm{x}^{k+1}-\bm{x}^{k} \\
				\bm{y}^{k+1} - \bm{y}^{k}
			\end{matrix}\right\|_{\Lambda_k}^{2}.
		\end{aligned}
	\end{equation*}
	Then, substituting this inequality into \eqref{gen-succ-xychgbar} yields that
	\begin{equation}\label{eq:dist-xybark}
		\begin{aligned}
			&\; \left\|\begin{matrix}
				\bm{x}^{k}-\overline{\bm{x}}^{k} \\
				\bm{y}^{k} - \overline{\bm{y}}^{k}
			\end{matrix} \right\|^{2}_{\Lambda_k}
			-
			\frac{1}{1+\nu_k}\left\|\begin{matrix}
				\bm{x}^{k+1}-\overline{\bm{x}}^{k+1} \\
				\bm{y}^{k+1} - \overline{\bm{y}}^{k+1}
			\end{matrix} \right\|^{2}_{\Lambda_{k+1}} \\
			\geq &\; \left(1 - \frac{2\kappa\sqrt{\tau_{k}\lambda_{\max}(S)}
				\left(\rho +\sqrt{\rho\overline{\tau}_{k}}\right)
				+ 2\sigma_{k}\sqrt{\rho}}{\sigma_{k}\sqrt{\tau_{k}\lambda_{\min}(S)}}\right)
			\left\|\begin{matrix}
				\bm{x}^{k+1}-\bm{x}^{k} \\
				\bm{y}^{k+1} - \bm{y}^{k}
			\end{matrix}\right\|_{\Lambda_k}^{2}.
		\end{aligned}
	\end{equation}
	
	\vspace{2mm}
	\textbf{Step III.} In the following, we will establish the asymptotic convergence rate based on \eqref{eq:dist-xybark}. First, by recalling the conditions on $\{\tau_{k}\}$: $\tau_{k}\geq\tau_{\min}>0$ and $\tau_{k+1}\leq(1+\nu_{k})\tau_{k}$ with $\nu_{k}\geq0$ and $\sum_{k=0}^{\infty}\nu_{k}<+\infty$, we know that there exists $\tau_{\max}:=\tau_{0}\prod_{k=0}^{\infty}(1+\nu_{k})$ such that $0<\tau_{\min} \leq \tau_{k} \leq \tau_{\max} < +\infty$ for all $k\geq 0$. This together with condition \eqref{para-conds} implies that there exists a positive integer $k_{0}$ such that
	\begin{equation*}
		\sqrt{\tau_{k}\lambda_{\min}(S)} - 2\sqrt{\rho} > 0
		\quad \text{and} \quad
		\sigma_{k} > c\cdot\frac{2\kappa\sqrt{\tau_{k}\lambda_{\max}(S)}\left(\rho+\sqrt{\rho\overline{\tau}_{k}}\right)}{\sqrt{\tau_{k}\lambda_{\min}(S)} - 2\sqrt{\rho}}, \quad \forall\,k \geq k_{0},
	\end{equation*}
	where $c>1$. Hence, one can verify that the following inequality holds for all $k\geq k_0$:
	\begin{equation}\label{para-condsre}
		\begin{aligned}
			& \left(1 - \frac{2\kappa\sqrt{\tau_{k}\lambda_{\max}(S)}
				\left(\rho +\sqrt{\rho\overline{\tau}_{k}}\right)
				+ 2\sigma_{k}\sqrt{\rho}}{\sigma_{k}\sqrt{\tau_{k}\lambda_{\min}(S)}}\right)
			>  \widetilde{c}:=\frac{c-1}{c}\cdot\frac{\sqrt{\tau_{\min}\lambda_{\min}(S)} - 2\sqrt{\rho}}{\sqrt{\tau_{\min}\lambda_{\min}(S)}}
			> 0,
		\end{aligned}
	\end{equation}
	%\begin{equation}\label{para-condsre}
	%\left(1 - \frac{2\kappa\sqrt{\tau_{k}}\left(\rho+\sqrt{\rho\overline{\tau}_{k}}\right)
		%+ 2\sigma_{k}\sqrt{\rho}}{\sigma_{k}\sqrt{\tau_{k}}}\right)
	%> \widetilde{c}:=\frac{c-1}{c}\cdot\frac{\sqrt{\tau_{\min}} - 2\sqrt{\rho}}{\sqrt{\tau_{\min}}}
	%> 0,
	%\quad \forall\,k \geq k_{0},
	%\end{equation}
	which means that the factor in the right-hand side of \eqref{eq:dist-xybark} will be \textit{positive} when $k\geq k_{0}$.
	
	On the other hand, using \eqref{gen-upbd-xdist} again, we deduce that
	\begin{equation*}
		\begin{aligned}
			\left\|\begin{matrix}
				\bm{x}^{k+1}-\bm{x}^{k} \\
				\bm{y}^{k+1} - \bm{y}^{k}
			\end{matrix}\right\|_{\Lambda_k}^{2}
			&\geq \frac{\sigma_{k}^{2}}{\kappa^2\left(\sqrt{\rho}+\sqrt{\overline{\tau}_{k}}\right)^2}
			\left\|\begin{matrix}
				\bm{x}^{k+1}-\widetilde{\bm{x}}^{k+1} \\
				\bm{y}^{k+1} - \overline{\bm{y}}^{k+1}
			\end{matrix}\right\|^{2}  \\
			&\geq \frac{\sigma_{k}^{2}}{\kappa^2\left(\sqrt{\rho}+\sqrt{\overline{\tau}_{k}}\right)^2\overline{\tau}_{k}}
			\left\|\begin{matrix}
				\bm{x}^{k+1}-\overline{\bm{x}}^{k+1} \\
				\bm{y}^{k+1} - \overline{\bm{y}}^{k+1}
			\end{matrix}\right\|_{\Lambda_k}^{2} \\
			&\geq \frac{1}{1+\nu_{k}}\cdot\frac{\sigma_{k}^{2}}{\kappa^2\left(\sqrt{\rho}+\sqrt{\overline{\tau}_{k}}\right)^2\overline{\tau}_{k}}
			\left\|\begin{matrix}
				\bm{x}^{k+1}-\overline{\bm{x}}^{k+1} \\
				\bm{y}^{k+1} - \overline{\bm{y}}^{k+1}
			\end{matrix}\right\|_{\Lambda_{k+1}}^{2}.
		\end{aligned}
	\end{equation*}
	This, together with \eqref{eq:dist-xybark} and \eqref{para-condsre}, yields that
	\begin{equation}\label{eq:linear_rate1}
		\left\|\begin{matrix}
			\bm{x}^{k}-\overline{\bm{x}}^{k} \\
			\bm{y}^{k} - \overline{\bm{y}}^{k}
		\end{matrix}\right\|_{\Lambda_k}^{2}
		\geq \left(\frac{1 + \gamma_{k}}{1 + \nu_{k}}\right)\left\|\begin{matrix}
			\bm{x}^{k+1}-\overline{\bm{x}}^{k+1} \\
			\bm{y}^{k+1} - \overline{\bm{y}}^{k+1}
		\end{matrix} \right\|_{\Lambda_{k+1}}^{2}, \quad \forall \,k \geq k_{0},
	\end{equation}
	where
	\begin{equation*}%\label{gammalbd}
		\begin{aligned}
			\gamma_{k} \,:=& \, \left(1 - \frac{2\kappa\sqrt{\tau_{k}\lambda_{\max}(S)}
				\left(\rho +\sqrt{\rho\overline{\tau}_{k}}\right)
				+ 2\sigma_{k}\sqrt{\rho}}{\sigma_{k}\sqrt{\tau_{k}\lambda_{\min}(S)}}\right) \frac{\sigma_{k}^{2}}{\kappa^2\left(\sqrt{\rho}+\sqrt{\overline{\tau}_{k}}\right)^2\overline{\tau}_{k}} \\
			\, \geq & \, \widetilde{c}\cdot\frac{\sigma_{k}^{2}}{\kappa^2\left(\sqrt{\rho}+\sqrt{\overline{\tau}_{\max}}\right)^2\overline{\tau}_{\max}}
			\geq \gamma_{\min}:= \frac{\widetilde{c}\,\sigma_{\min}^{2}}{\kappa^2\left(\sqrt{\rho}+\sqrt{\overline{\tau}_{\max}}\right)^2\overline{\tau}_{\max}} > 0, \quad \forall\,k \geq k_{0},
		\end{aligned}
	\end{equation*}
	and $\overline{\tau}_{\max}:=\max\left\{1,\tau_{\max}\lambda_{\max}(S)\right\}$. Then, one can readily obtain from \eqref{eq:linear_rate1} that
	\begin{equation*}
		\mathrm{dist}_{\Lambda_{k+1}}\left((\bm{x}^{k+1},\bm{y}^{k+1}), \,(\partial\ell)^{-1}(\bm{0},\bm{0})\right) \leq \mu_{k}\, \mathrm{dist}_{\Lambda_k}\left((\bm{x}^{k},\bm{y}^{k}), \,(\partial\ell)^{-1}(\bm{0},\bm{0})\right),
	\end{equation*}
	where $\mu_{k}:=\sqrt{\frac{1+\nu_{k}}{1+\gamma_{k}}}$. Since $\nu_{k}\to0$ and $\gamma_{k}\geq\gamma_{\min}> 0$ for all $k \geq k_{0}$, one can verify that $\limsup\limits_{k\to\infty}\,\{\mu_{k}\}<1$. Thus, we obtain the desired results in statement (i).
	
	\vspace{2mm}
	\textit{Statement (ii)}. Using \eqref{eq:dist-xybark} and \eqref{para-condsre} again, we see that
	\begin{equation*}
		\begin{aligned}
			\widetilde{c}\,\tau_{\min}\lambda_{\min}(S)\|\bm{x}^{k+1}-{\bm{x}}^{k}\|^2
			& \; \leq \widetilde{c}\tau_{\min}\|\bm{x}^{k+1}-{\bm{x}}^{k}\|^2_{S} \\
			& \; \leq
			\widetilde{c}\left\|\begin{matrix}
				\bm{x}^{k+1}-{\bm{x}}^{k} \\
				\bm{y}^{k+1} - {\bm{y}}^{k}
			\end{matrix}\right\|_{\Lambda_k}^{2}
			\leq \mathrm{dist}_{\Lambda^{k}}^{2}\left((\bm{x}^{k},\bm{y}^{k}), \,(\partial\ell)^{-1}(\bm{0},\bm{0})\right),
		\end{aligned}
	\end{equation*}
	holds for any $ k \geq k_{0}$. Using this inequality and the fact that $\left\{\mathrm{dist}_{\Lambda^{k}}\left((\bm{x}^{k},\bm{y}^{k}), \,(\partial\ell)^{-1}(\bm{0},\bm{0})\right)\right\}$ is asymptotic Q-(super)linear convergent, we can conclude that there exist a positive integer $k_1$, $0<\beta<1$ and $C>0$ such that
	\begin{equation*}
		\|\bm{x}^{k+1}-{\bm{x}}^{k}\| \leq C \beta^{k}, \quad \forall \,k \geq k_{1},
	\end{equation*}
	which further implies that $\sum_{k=0}^{\infty}\|\bm{x}^{k+1}-{\bm{x}}^{k}\|<\infty$. Consequently, $\{\bm{x}^k\}$ is a Cauchy sequence and hence convergent. Therefore, the proof is completed.
\end{proof}

%%%%%%%%%%%%%%%%%%%%%%%%%%%%%%%%%%%%%%%%%%
\section{Proof of Theorem \ref{thm:ergo-rate}}\label{sec:proof-ergo-rate}

\begin{proof}
	First, it follows from the definition of $\widehat{\bm{x}}^k$ and the updating rule of $\bm{\lambda}^{k+1}$ in \eqref{rip2ALM_ywupdate} that
	\begin{equation}\label{eq-vio}
		\begin{aligned}
			\|A\widehat{\bm{x}}^k-\bm{b}\|
			&\,=\,\left\|A\left(\frac{\sum_{i=0}^{k-1} \sigma_i\bm{x}^{i+1}}{\sum_{i=0}^{k-1}\sigma_{i}}\right)-\bm{b}\right\|
			\,=\,\frac{1}{\sum_{i=0}^{k-1}\sigma_{i}}\left\|
			\sum_{i=0}^{k-1}\sigma_{i}\left(A\bm{x}^{i+1}-\bm{b}\right)\right\|  \\
			&\,=\,\frac{1}{\sum_{i=0}^{k-1}\sigma_{i}}\left\|
			\sum_{i=0}^{k-1}\left(\bm{\lambda}^{i+1}-\bm{\lambda}^{i}\right)\right\|
			\,=\,\frac{1}{\sum_{i=0}^{k-1}\sigma_{i}}\left\|\bm{\lambda}^{k}-\bm{\lambda}^{0}\right\|.
		\end{aligned}
	\end{equation}
	Next, note that $g(\bm{x}) = \big(g_1(\bm{x}), \cdots, g_{m_{2}}(\bm{x})\big)$ and each $g_j:\mathbb{R}^{N}\to\mathbb{R}$ ($j=1,\cdots,m_2$) is convex. Then, for an arbitrary index $j\in\{1,\cdots,m_2\}$, it follows from the convexity of $g_{j}$ and the updating rule of $\bm{\mu}^{k+1}$ in \eqref{rip2ALM_ywupdate} that
	\begin{equation*}
		\begin{aligned}
			g_{j}(\widehat{\bm{x}}^k) & \, = g_{j}\left(\frac{\sum_{i=0}^{k-1} \sigma_{i}\bm{x}^{i+1}}{\sum_{i=0}^{k-1}\sigma_{i}}\right) \leq \frac{\sum_{i=0}^{k-1} \sigma_{i}g_{j}(\bm{x}^{i+1})}{\sum_{i=0}^{k-1}\sigma_{i}} = \frac{\sum_{i=0}^{k-1} \big(\mu_{j}^{i}+\sigma_{i}g_{j}(\bm{x}^{i+1})-\mu_{j}^{i}\big)}{\sum_{i=0}^{k-1}\sigma_{i}} \\[5pt]
			& \, \leq \frac{\sum_{i=0}^{k-1} \big( \max\{0, \mu_{j}^{i}+\sigma_{i}g_{j}(\bm{x}^{i+1})\} - \mu_{j}^{i} \big)}{\sum_{i=0}^{k-1}\sigma_{i}} = \frac{\sum_{i=0}^{k-1} \big(\mu_{j}^{i+1}-\mu_{j}^{i}\big)}{\sum_{i=0}^{k-1}\sigma_{i}} = \frac{\mu_{j}^{k}-\mu_{j}^{0}}{\sum_{i=0}^{k-1}\sigma_{i}},
		\end{aligned}
	\end{equation*}
	where $\mu_{j}^{i}$ denotes the $j$-th element of $\bm{\mu}^{i}$.
	%, and the second inequality follows from the fact that $a\leq \max\{0,a\}$ for any $a\in\mathbb{R}$.
	This inequality further yields that
	\begin{equation*}
		\max\big\{0,\,g_{j}(\widehat{\bm{x}}^k)\big\} \leq \; \frac{1}{\sum_{i=0}^{k-1}\sigma_{i}}\big|\mu_{j}^{k}-\mu_{j}^{0}\big|, \quad \forall \, j = 1,\cdots, m_2,
	\end{equation*}
	which implies that
	\begin{equation}\label{ineq-vio}
		\left\|\max\left\{\bm{0},\,g(\widehat{\bm{x}}^k)\right\}\right\|
		\leq \frac{1}{\sum_{i=0}^{k-1}\sigma_{i}}\left\|\bm{\mu}^{k} - \bm{\mu}^{0}\right\|.
	\end{equation}
	Thus, combining \eqref{eq-vio} and \eqref{ineq-vio}, we obtain that
	\begin{equation*}
		\begin{Vmatrix}
			A \widehat{\bm{x}}^k-\bm{b} \\
			\max\left\{\bm{0},\,g(\widehat{\bm{x}}^k)\right\}
		\end{Vmatrix}
		\leq \frac{1}{\sum_{i=0}^{k-1}\sigma_{i}}\left\|\bm{y}^{k}-\bm{y}^{0}\right\|
		\leq \frac{2B_y}{\sum_{i=0}^{k-1}\sigma_{i}}.
	\end{equation*}
	This proves inequality \eqref{primfeas-vio}.
	
	We now proceed to prove the left-hand side of inequality \eqref{primobj-gap}. Since $ (\bm{x}^{*},\,\bm{y}^*)$, with $\bm{y}^{*}:=(\bm{\lambda}^{*},\bm{\mu}^{*})$, is a saddle point to $\ell$, it is known that (see, e.g., \cite[Theorem 11.59]{rw1998variational})
	\begin{equation*}
		f(\bm{x}^{*}) = \min_{\bm{x}} ~ \mathcal{L}_{\sigma_{k}}(\bm{x}, \bm{\lambda}^*, \bm{\mu}^*), \quad \forall\,\sigma_{k}>0.
	\end{equation*}
	Therefore, for any $\bm{x}^{\prime}\in\mathbb{R}^{N}$, we have
	\begin{align}
		f(\bm{x}^{*})
		= & \; \min_{\bm{x}} ~ \mathcal{L}_{\sigma_{k}}(\bm{x}, \bm{\lambda}^*, \bm{\mu}^*) \nonumber \\
		\leq & \; f(\bm{x}') + \langle\bm{\lambda}^*, A\bm{x}'-\bm{b}\rangle
		+ \frac{\sigma_k}{2}\|A\bm{x}'-\bm{b}\|^{2}
		+ \frac{1}{2\sigma_k}\big\|\max\{\bm{0},\bm{\mu}^*+\sigma_kg(\bm{x}')\}\big\|^{2}
		- \frac{1}{2\sigma_k}\|\bm{\mu}^*\|^{2} \nonumber \\
		\leq & \; f(\bm{x}') + \langle\bm{\lambda}^*, A\bm{x}'-\bm{b}\rangle
		+ \frac{\sigma_{k}}{2}\|A\bm{x}'-\bm{b}\|^{2}
		+ \frac{1}{2\sigma_{k}}\big\|\bm{\mu}^*+\max\{\bm{0},\sigma_{k} g(\bm{x}^{\prime})\}\big\|^{2}-\frac{1}{2\sigma_{k}}\|\bm{\mu}^*\|^{2} \nonumber \\
		= & \; f(\bm{x}^{\prime}) + \Bigg\langle\begin{bmatrix}
			\bm{\lambda}^{*} \\
			\bm{\mu}^{*}
		\end{bmatrix},
		\begin{bmatrix}
			A\bm{x}^{\prime}-\bm{b} \\ \max\{\bm{0}, g(\bm{x}^{\prime})\}
		\end{bmatrix}\Bigg\rangle + \frac{\sigma_{k}}{2}\begin{Vmatrix}
			A\bm{x}^{\prime}-\bm{b} \\ \max\{\bm{0}, g(\bm{x}^{\prime})\}
		\end{Vmatrix}^{2} \nonumber  \\
		\leq & \; f(\bm{x}^{\prime}) + \|\bm{y}^{*}\| \begin{Vmatrix}
			A\bm{x}^{\prime}-\bm{b} \\ \max\{\bm{0}, g(\bm{x}^{\prime})\}
		\end{Vmatrix} + \frac{\sigma_{k}}{2}\begin{Vmatrix}
			A\bm{x}^{\prime}-\bm{b} \\ \max\{\bm{0}, g(\bm{x}^{\prime})\}
		\end{Vmatrix}^{2}, \label{primobj-gap-left1}
	\end{align}
	where the second inequality follows from the fact that $\big\|\max\{\bm{0},\bm{\mu}^*+\sigma_{k} g(\bm{x}^{\prime})\}\big\|^{2}\leq\big\|\bm{\mu}^*+\max\{\bm{0},\sigma_{k} g(\bm{x}^{\prime})\}\big\|^{2}$. Substituting $\bm{x}^{\prime} = \widehat{\bm{x}}^{k}$ into \eqref{primobj-gap-left1} and using the estimate in \eqref{primfeas-vio} yields the desired bound the left-hand side of \eqref{primobj-gap}.
	
	We finally prove the right-hand side of inequality \eqref{primobj-gap}. For simplicity, let
	\begin{equation*}
		\Psi_{k}(\bm{x}) := \mathcal{L}_{\sigma_k}(\bm{x},\bm{\lambda}^k,\bm{\mu}^k) + \frac{\tau_{k}}{2\sigma_{k}}\|\bm{x}-\bm{x}^{k}\|_{S}^{2}.
	\end{equation*}
	Due to the symmetric positive definiteness of the preconditioner $S$, $\Psi_{k}$ is $\frac{\tau_{k}}{\sigma_{k}}$-strongly convex with respect to the weighted norm $\|\cdot\|_{S}$. Moreover, from \eqref{rip2ALM-inexcond}, we know that $\Delta^{k+1}\in\partial\Psi_k(\bm{x}^{k+1})$. Therefore, by the strong convexity of $\Psi_{k}$, we have
	\begin{equation*}
		\Psi_{k}(\bm{x}^{*}) \geq \Psi_{k}(\bm{x}^{k+1}) + \langle\Delta^{k+1} \; ,\bm{x}^{*}-\bm{x}^{k+1}\rangle+\frac{\tau_{k}}{2\sigma_{k}}\|\bm{x}^{*}-\bm{x}^{k+1}\|^2_{S}.
	\end{equation*}
	Expanding this inequality and recalling the definitions of $\mathcal{L}_{\sigma_k}$ and $\Psi_k$, together with the fact that $\bm{x}^*$ is feasible, we obtain that
	\begin{align}
		&\; f(\bm{x}^{*})
		+ \frac{1}{2\sigma_{k}}\big\|\max\{\bm{0},\bm{\mu}^{k}+\sigma_{k} g(\bm{x}^{*})\}\big\|^{2}-\frac{1}{2\sigma_{k}}\|\bm{\mu}^{k}\|^{2}
		+\frac{\tau_{k}}{2\sigma_{k}}\|\bm{x}^{*}-\bm{x}^{k}\|^2_{S}  \nonumber \\
		\geq & \; f(\bm{x}^{k+1}) + \langle\bm{\lambda}^{k},A\bm{x}^{k+1}-\bm{b}\rangle + \frac{\sigma_{k}}{2}\|A\bm{x}^{k+1}-\bm{b}\|^{2}+\frac{1}{2\sigma_{k}}\big\|\max\{\bm{0},\bm{\mu}^{k}+\sigma_{k} g(\bm{x}^{k+1})\}\big\|^{2} \nonumber \\
		& \; - \frac{1}{2\sigma_{k}}\|\bm{\mu}^{k}\|^{2}
		+ \langle\Delta^{k+1},\,\bm{x}^{*}-\bm{x}^{k+1}\rangle
		+ \frac{\tau_{k}}{2\sigma_{k}}\|\bm{x}^{k+1}-\bm{x}^{k}\|_{S}^{2}
		+ \frac{\tau_{k}}{2\sigma_{k}}\|\bm{x}^{*}-\bm{x}^{k+1}\|_{S}^{2} \nonumber \\
		\geq & \; f(\bm{x}^{k+1}) + \frac{1}{2\sigma_{k}}\|\bm{\lambda}^{k+1}\|^{2} - \frac{1}{2\sigma_{k}}\|\bm{\lambda}^{k}\|^{2} + \frac{1}{2\sigma_{k}}\|\bm{\mu}^{k+1}\|^{2} - \frac{1}{2\sigma_{k}}\|\bm{\mu}^{k}\|^{2} \nonumber \\
		& \; + \langle\Delta^{k+1},\,\bm{x}^{*}-\bm{x}^{k+1}\rangle + \frac{\tau_{k}}{2\sigma_{k}}\|\bm{x}^{*}-\bm{x}^{k+1}\|_{S}^{2}, \nonumber \\
		= & \; f(\bm{x}^{k+1}) + \frac{1}{2\sigma_{k}}\|\bm{y}^{k+1}\|^{2} - \frac{1}{2\sigma_{k}}\|\bm{y}^{k}\|^{2} + \langle\Delta^{k+1},\,\bm{x}^{*}-\bm{x}^{k+1}\rangle + \frac{\tau_{k}}{2\sigma_{k}}\|\bm{x}^{*}-\bm{x}^{k+1}\|_{S}^{2}, \label{pAL-1stcond}
	\end{align}
	where the second inequality follows from the updating rules of $\bm{\lambda}^{k+1}$ and $\bm{\mu}^{k+1}$ in \eqref{rip2ALM_ywupdate}. Then, recalling \eqref{gen-wwineq} and $\bm{p}^{k+1}:=\Delta^{k+1}-\tau_k\sigma_{k}^{-1}S(\bm{x}^{k+1}-\bm{x}^{k})$,  we have that
	\begin{equation}\label{gen-wwineq2}
		\begin{aligned}
			\|\bm{w}^k-\bm{x}^*\|^2
			= &\; \|\bm{w}^{k+1}-\bm{x}^*\|^2 + 2\langle\bm{w}^k-\bm{x}^{k+1},\,\sigma_{k}\Delta^{k+1}\rangle - \|\sigma_{k}\Delta^{k+1}\|^2 \\
			&\; + 2\sigma_{k}\langle\bm{x}^{k+1}-\bm{x}^*,\, \Delta^{k+1}\rangle.
		\end{aligned}
	\end{equation}
	Multiplying both sides of \eqref{gen-wwineq2} by $\frac{1}{2\sigma_{k}}$ and adding it to the both sides of \eqref{pAL-1stcond}, we have that
	\begin{align}
		&\; f(\bm{x}^{*})
		+ \frac{1}{2\sigma_{k}}\big\|\max\{\bm{0},\bm{\mu}^{k}+\sigma_{k} g(\bm{x}^{*})\}\big\|^{2}-\frac{1}{2\sigma_{k}}\|\bm{\mu}^{k}\|^{2}
		+\frac{\tau_{k}}{2\sigma_{k}}\|\bm{x}^{*}-\bm{x}^{k}\|^2_{S} + \frac{1}{2\sigma_{k}}\|\bm{w}^{k}-\bm{x}^*\|^2 \nonumber \\[2mm]
		\geq & \; f(\bm{x}^{k+1}) + \frac{1}{2\sigma_{k}}\|\bm{y}^{k+1}\|^{2} - \frac{1}{2\sigma_{k}}\|\bm{y}^{k}\|^{2}
		+ \langle\Delta^{k+1},\,\bm{x}^{*}-\bm{x}^{k+1}\rangle
		+ \frac{\tau_{k}}{2\sigma_{k}}\|\bm{x}^{*}-\bm{x}^{k+1}\|_{S}^{2} \nonumber \\
		& \; + \frac{1}{2\sigma_{k}}\left(\|\bm{w}^{k+1}-\bm{x}^*\|^2 + 2\langle\bm{w}^k-\bm{x}^{k+1},\,\sigma_{k}\Delta^{k+1}\rangle - \|\sigma_{k}\Delta^{k+1}\|^2 \right) + \langle\bm{x}^{k+1}-\bm{x}^*,\, \Delta^{k+1}\rangle \nonumber \\[2mm]
		= & \; f(\bm{x}^{k+1}) + \frac{1}{2\sigma_{k}}\|\bm{y}^{k+1}\|^{2} - \frac{1}{2\sigma_{k}}\|\bm{y}^{k}\|^{2} + \frac{\tau_{k}}{2\sigma_{k}}\|\bm{x}^{*}-\bm{x}^{k+1}\|_{S}^{2} + \frac{1}{2\sigma_{k}}\|\bm{w}^{k+1}-\bm{x}^*\|^2 \nonumber \\
		&
		\; + \frac{1}{2\sigma_{k}}\left( 2\langle\bm{w}^k-\bm{x}^{k+1},\,\sigma_{k}\Delta^{k+1}\rangle - \|\sigma_{k}\Delta^{k+1}\|^2 \right), \label{pAL-1stcond2}
	\end{align}

	\noindent Moreover, since $g(\bm{x}^*)\leq0$, we have $\big\|\max\{\bm{0},\bm{\mu}^{k}+\sigma_{k} g(\bm{x}^{*})\}\big\|^{2}\leq\|\bm{\mu}^{k}\|^{2}$. This, together with \eqref{pAL-1stcond2}, implies that
	\begin{equation*}
		\begin{aligned}
			&\quad \sigma_k\big(f(\bm{x}^{k+1}) - f(\bm{x}^*)\big)  \\
			&\leq \frac{\tau_k}{2}\|\bm{x}^{*}-\bm{x}^{k}\|^2_{S}
			- \frac{\tau_k}{2}\|\bm{x}^{*}-\bm{x}^{k+1}\|^2_{S}
			+ \frac{1}{2}\|\bm{y}^{k}\|^{2} - \frac{1}{2}\|\bm{y}^{k+1}\|^{2}
			\\
			& \qquad + \frac{1}{2}\|\bm{w}^{k} - \bm{x}^{*}\|^{2} - \frac{1}{2}\|\bm{w}^{k+1} - \bm{x}^{*}\|^{2} + \frac{1}{2}\left( - 2\langle\bm{w}^k-\bm{x}^{k+1},\,\sigma_{k}\Delta^{k+1}\rangle + \|\sigma_{k}\Delta^{k+1}\|^2 \right) \\
			%&\leq \frac{\tau_k}{2}\|\bm{x}^{*}-\bm{x}^{k}\|^2_{S}
			%- \frac{\tau_{k+1}}{2(1+\nu_k)}\|\bm{x}^{*}-\bm{x}^{k+1}\|^2_{S}
			%+ \frac{1}{2}\|\bm{y}^{k}\|^{2} - \frac{1}{2}\|\bm{y}^{k+1}\|^{2}
			%\\
			%& \qquad + \frac{1}{2}\|\bm{w}^{k} - \bm{x}^{*}\|^{2} - \frac{1}{2}\|\bm{w}^{k+1} - \bm{x}^{*}\|^{2} + \frac{1}{2}\left( - 2\langle\bm{w}^k-\bm{x}^{k+1},\,\sigma_{k}\Delta^{k+1}\rangle + \|\sigma_{k}\Delta^{k+1}\|^2 \right) \\
			&\leq \frac{\tau_k}{2}\|\bm{x}^{*}-\bm{x}^{k}\|^2_{S}
			- \frac{\tau_{k+1}}{2(1+\nu_k)}\|\bm{x}^{*}-\bm{x}^{k+1}\|^2_{S}
			+ \frac{1}{2}\|\bm{y}^{k}\|^{2} - \frac{1}{2}\|\bm{y}^{k+1}\|^{2}
			\\
			& \qquad + \frac{1}{2}\|\bm{w}^{k} - \bm{x}^{*}\|^{2} - \frac{1}{2}\|\bm{w}^{k+1} - \bm{x}^{*}\|^{2} + \frac{1}{2}\left( 2\big|\langle\bm{w}^k-\bm{x}^{k+1},\,\sigma_{k}\Delta^{k+1}\rangle\big| + \|\sigma_{k}\Delta^{k+1}\|^2 \right) \\
			&\leq \frac{\tau_k}{2}\|\bm{x}^{*}-\bm{x}^{k}\|^2_{S}
			- \frac{\tau_{k+1}}{2}\|\bm{x}^{*}-\bm{x}^{k+1}\|^2_{S} + \frac{\nu_k\tau_{k+1}}{2(1+\nu_k)}\|\bm{x}^{*}-\bm{x}^{k+1}\|^2_{S}
			+ \frac{1}{2}\|\bm{y}^{k}\|^{2} - \frac{1}{2}\|\bm{y}^{k+1}\|^{2}  \\
			&\qquad + \frac{1}{2}\|\bm{w}^{k} - \bm{x}^{*}\|^{2} - \frac{1}{2}\|\bm{w}^{k+1} - \bm{x}^{*}\|^{2} + \frac{\rho}{2}\left( \big\|\bm{y}^{k+1}-\bm{y}^{k}\big\|^2
			+ \tau_k\big\|\bm{x}^{k+1}-\bm{x}^{k}\big\|^{2}_{S} \right)  \\
			&\leq \frac{\tau_k}{2}\|\bm{x}^{*}-\bm{x}^{k}\|^2_{S}
			- \frac{\tau_{k+1}}{2}\|\bm{x}^{*}-\bm{x}^{k+1}\|^2_{S}
			+ \lambda_{\max}(S)(\|\bm{x}^*\|^2+B_x^2)\tau_{\max}\nu_k + \frac{1}{2}\|\bm{y}^{k}\|^{2}-\frac{1}{2}\|\bm{y}^{k+1}\|^{2} \\
			&\qquad
			+ \frac{1}{2}\|\bm{w}^{k} - \bm{x}^{*}\|^{2} - \frac{1}{2}\|\bm{w}^{k+1} - \bm{x}^{*}\|^{2} + \frac{\rho}{2}\left( \big\|\bm{y}^{k+1}-\bm{y}^{k}\big\|^2
			+ \tau_k\big\|\bm{x}^{k+1}-\bm{x}^{k}\big\|^{2}_{S} \right),
		\end{aligned}
	\end{equation*}
	where the second inequality follows from $\tau_{k+1}\leq(1+\nu_{k})\tau_{k}$, the third inequality follows from the error criterion \eqref{gen-optcond2}, and the last inequality follows from $\nu_k\geq0$, $\tau_k\leq\tau_{\max}$, and $\|\bm{x}^{*}-\bm{x}^{k+1}\|^2_{S}\leq2(\|\bm{x}^*\|^2_{S}+\|\bm{x}^{k+1}\|^2_{S})
	\leq2\lambda_{\max}(S)(\|\bm{x}^*\|^2+B_x^2)$ with $B_x$ being the upper bound of the sequence $\{\bm{x}^k\}$. Then, from the above inequality, we obtain that
	\begin{equation*}
		\begin{aligned}
			&\quad \sum_{i=0}^{k-1}\sigma_{i}\big(f(\bm{x}^{i+1}) - f(\bm{x}^{*})\big) \\
			&\leq \; \sum_{i=0}^{k-1}\left(\frac{\tau_i}{2}\|\bm{x}^{*}-\bm{x}^{i}\|^2_{S}
			- \frac{\tau_{i+1}}{2}\|\bm{x}^{*}-\bm{x}^{i+1}\|^2_{S}\right)
			+ \sum_{i=0}^{i-1}\left(\frac{1}{2}\|\bm{y}^{i}\|^{2}
			-\frac{1}{2}\|\bm{y}^{i+1}\|^{2}\right)   \\
			&\qquad
			+ \sum_{i=0}^{k-1}\left(\frac{1}{2}\|\bm{w}^{i} - \bm{x}^{*}\|^{2} - \frac{1}{2}\|\bm{w}^{i+1} - \bm{x}^{*}\|^{2}\right) + \lambda_{\max}(S)(\|\bm{x}^*\|^2+B_x^2)\tau_{\max}\sum_{i=0}^{k-1}\nu_i \\
			&\qquad + \frac{\rho}{2}\sum_{i=0}^{k-1}\left( \big\|\bm{y}^{i+1}-\bm{y}^{i}\big\|^2
			+ \tau_{i}\big\|\bm{x}^{i+1}-\bm{x}^{i}\big\|^{2}_{S} \right)  \\
			&\leq \; \frac{\tau_{0}}{2}\|\bm{x}^{*}-\bm{x}^{0}\|_{S}^{2}
			+ \frac{1}{2}\|\bm{y}^{0}\|^{2} + \frac{1}{2}\|\bm{w}^{0} - \bm{x}^{*}\|^{2}
			+ \lambda_{\max}(S)(\|\bm{x}^*\|^2+B_x^2)\tau_{\max}\sum_{i=0}^{\infty}\nu_i \\
			&\qquad + \frac{\rho}{2}\sum_{i=0}^{\infty}\left( \big\|\bm{y}^{i+1}-\bm{y}^{i}\big\|^2
			+ \tau_{i}\big\|\bm{x}^{i+1}-\bm{x}^{i}\big\|^{2}_{S} \right).
		\end{aligned}
	\end{equation*}
	Moreover, it follows from the convexity of $f$ that
	\begin{equation*}
		f(\widehat{\bm{x}}^{k}) - f(\bm{x}^{*})
		= f\left(\frac{\sum_{i=0}^{k-1}\sigma_{i}\bm{x}^{i+1}}{\sum_{i=0}^{k-1}\sigma_{i}}\right) - f(\bm{x}^{*})
		\leq \cfrac{\sum_{i=0}^{k-1}\sigma_{i}\big(f(\bm{x}^{i+1}) - f(\bm{x}^{*})\big)}{\sum_{i=0}^{k-1}\sigma_{i}}.
	\end{equation*}
	Combining the above two results yields the desired bound the right-hand side of \eqref{primobj-gap}.
\end{proof}

%%%%%%%%%%%%%%%%%%%%%%%%% References
\bibliographystyle{plain}
\bibliography{Ref_ripALM_general}

\end{document}